\documentclass[a4paper,12pt,reqno]{amsart}
\usepackage{amssymb,amsthm}
\usepackage{amsmath}
\usepackage{mathtools}
\usepackage{ifthen}
\usepackage{graphicx}
\usepackage{xcolor}
\usepackage{float}
\usepackage{tcolorbox}
\usepackage{caption}
\usepackage{subcaption}
\usepackage{diagbox}
\theoremstyle{plain}
\newtheorem{assumption}{Assumption}
\newtheorem{thm}{Theorem}[section]
\newtheorem*{thm*}{Theorem}

\numberwithin{equation}{section}

\numberwithin{subcase}{case}
\newtheorem{cor}{Corollary}[section]
\newtheorem{lem}{Lemma}[section]
\newtheorem{prop}{Proposition}[section]

\newtheorem{defn}{Definition}[section]

\theoremstyle{definition}
\newcounter {own}
\def\theown {\thesection  .\arabic{own}}

\newenvironment{pf}[1][]{%
 \vskip 3mm
 \noindent
 \ifthenelse{\equal{#1}{}}%
  {{\slshape Proof. }}%
  {{\slshape #1.} }%
 }%
{\qed\bigskip}

\newcounter{alphabet}

\newcommand{\ds}{\displaystyle}

\newcounter{minutes}
\divide\time by 60
\newcounter{hours}
\multiply\time by 60 \addtocounter{minutes}{-\time}

\begin{document}
\bibliographystyle{amsplain}
\title{Derivative sampling expansions in shift-invariant spaces with error estimates covering discontinuous signals}

\thanks{
File:~AntonyPriya3.tex,
          printed: 2022-02-20,
          \thehours.\ifnum\theminutes<10{0}\fi\theminutes}

\author{Kumari Priyanka}
\author{A. Antony Selvan$^\dagger$}
\address{Kumari Priyanka, Indian Institute of Technology Dhanbad, Dhanbad-826 004, India.}
\email{priyankak4193@gmail.com}
\address{A. Antony Selvan, Indian Institute of Technology Dhanbad, Dhanbad-826 004, India.}
\email{antonyaans@gmail.com}

\keywords{B-Splines, Laurent operators, Riesz basis, derivative sampling, shift-invariant spaces, averaged moduli of smoothness. \\
$^\dagger$ {\tt Corresponding author}
}
\maketitle
\pagestyle{myheadings}
\markboth{Kumari Priyanka and A. Antony Selvan}{Derivative sampling expansions in shift-invariant spaces with error estimates .....}

\begin{abstract}
This paper is concerned with the problem of sampling and interpolation involving derivatives in shift-invariant spaces and the error analysis of the derivative sampling expansions for fundamentally large classes of functions.  
A new type of polynomials based on derivative samples is introduced, which is different from the Euler-Frobenius polynomials for the multiplicity $r>1$. A complete characterization of uniform sampling with derivatives is given using Laurent operators.
The rate of approximation of a signal (not necessarily continuous) by the derivative sampling expansions in shift-invariant spaces generated by compactly supported functions is established in terms of 
$L^p$- average modulus of smoothness. Finally, several typical examples illustrating the various problems are discussed in detail.
\end{abstract}
\section{Introduction}
Let $\mathcal{B}_{\sigma}$ denote the space of bandlimited signals of bandwidth $\sigma,$ \textit{i.e.},
\begin{eqnarray*}
\mathcal{B}_{\sigma}=\left\{f\in L^2(\mathbb{R}):  \textrm{supp}\widehat{f}\subseteq [-\sigma,\sigma]\right\}.
\end{eqnarray*}
We will use the following version of the
Fourier transform: 
$$\widehat f(w)= \int_{-\infty}^{\infty} f(t) e^{-2\pi \mathrm{i}wt}dt,~ w\in\mathbb{R}.$$
The fundamental Whittaker-Shannon-Kotel'nikov sampling theorem states that every $f\in\mathcal{B}_{\pi W}$ can be reconstructed from its samples $\{f(n/W):n\in\mathbb{Z}\}$ by the sampling formula
\begin{eqnarray}
f(t)=\ds\sum_{n\in\mathbb{Z}}f\left(\frac{n}{W}\right)\text{sinc}(n-Wt),~\textrm{sinc}   t=\dfrac{\sin \pi t}{\pi t}.
\end{eqnarray}

The space of bandlimited signals is frequently unsuitable for numerical implementations because the sinc function has infinite support and decays very slowly. A shift-invariant space  $V(\phi)$ can be used as a universal model for sampling and interpolation problems because it can include a wide range of functions, whether bandlimited or not, by appropriately choosing a stable generator $\phi.$ The shift-invariant space $V(\phi)$ is defined as 
$$V(\phi):=\left\{f\in L^2(\mathbb{R}): f(x)=\sum\limits_{k\in\mathbb{Z}}c_k\phi(x-k)~\text{for}~ (c_k)\in \ell^2(\mathbb{Z})\right\}.$$
Recall that the generator $\phi$ is said to be stable if $\{\phi(\cdot-n) : n \in \mathbb{Z}\}$ is a Riesz basis for $V(\phi)$, \textit{i.e.},
$\overline{span\{\phi(\cdot-n):n\in\mathbb{Z}\}}=V(\phi)$ and there exist constants $0< A
\leq B<\infty$ such that
\begin{equation*}
A\sum_{n\in\mathbb{Z}}|d_n|^2\leq\big\|\sum_{n\in\mathbb{Z}}d_n\phi(\cdot-n)\big\|^2_{L^2(\mathbb{R})}\leq
B\sum_{n\in\mathbb{Z}}|d_n|^2,
\end{equation*}
for all $(d_n)\in\ell^2(\mathbb{Z})$. 
In addition, if $\phi$ is a continuously $r$-times differentiable function such that
for some 
$\epsilon>0$, $$\phi^{(s)}(t)=\mathcal{O}(|t|^{-0.5-\epsilon})~ \text{as}~ t\to\pm\infty,~ s=0,1,\dots,r,$$
then we say that $\phi$ is an $r$-regular stable generator for $V(\phi)$. In this case, every $f\in V(\phi)$ is $r$-times differentiable and 
\begin{eqnarray*}
f^{(s)}(t)=\sum\limits_{k\in\mathbb{Z}}c_k\phi^{(s)}(t-k),~0\leq s\leq r.
\end{eqnarray*}

The concept of shift-invariant spaces initially appeared in approximation theory and wavelet theory, which generalise the space of bandlimited signals. Sampling in non-bandlimited shift-invariant spaces is a suitable and realistic model for many applications, including taking into account real acquisition and reconstruction devices, modelling signals with smoother spectrum than bandlimited signals, and numerical implementation. We refer to \cite{AlGr} and the references therein for more information.

In several applications, such as aircraft instrument communications, air traffic management simulation, wideband fractional delay filter design \cite{Tseng}, or telemetry \cite{Fogel}, the possibility of obtaining sampling expansion involving sample values of a signal and its derivatives can be considered.
Fogel \cite{Fogel}, Jagermann and Fogel \cite{JagermanFogel}, Linden and Abramson \cite{LindenAbramson} proposed the first work on sampling with derivative values in the space of bandlimited signals.
They proved that if the values of $f\in\mathcal{B}_{\pi W}$ and its first $\rho-1$ derivatives are known on the sequence $\left\{\frac{\rho n}{W}:n\in\mathbb{Z}\right\},$ then $f$ can be reconstructed via the sampling formula
\begin{eqnarray}\label{ssf}
f(t)=\ds\sum_{n\in\mathbb{Z}}\sum_{k=0}^{\rho-1}
f_k\left(t_n\right)\dfrac{(t-t_n)^k}{k!}\left[\textrm{sinc}\tfrac{W}{\rho}(t-t_n)\right]^{\rho},
\end{eqnarray}
where $t_n=\frac{\rho n}{W}$ and $f_k(t_n)$ are linear combinations of $f(t_n), f^\prime(t_n),\dots,f^{(k)}(t_n):$

$$f_k(t_n):=\sum_{i=0}^k\binom{k}{i}\left(\frac{\pi W}{\rho}\right)^{k-i}\left[\dfrac{d^{k-i}}{dt^{k-i}}\left(\dfrac{t}{\sin t}\right)^{\rho}\right]\Bigg|_{t=0}f^{(i)}(t_n).$$

The authors \cite{AnAs1} studied the estimates for the truncation, amplitude, and jitter type errors of the sampling series \eqref{ssf} for $\rho=2.$
In \cite{AsAl}, the authors investigated the truncation error bounds of the derivative sampling series \eqref{ssf}. After the error analysis has been established, the authors in \cite{AnAs2, AnTh, ThAl} used the derivative sampling series to compute the eigenvalues of Strum-Liouville problems numerically.

Concerning the generalised sampling, suppose that $s$ linear time-invariant systems $\Upsilon_j$, $j = 1, 2,\dots,s$,
are defined on the shift-invariant space $V(\phi)$. 
A generalised sampling formula in $V(\phi)$ is a sampling expansion that allows any function $f\in V(\phi)$ to be recovered from the generalised samples $\{\Upsilon_jf(\rho l):l\in\mathbb{Z}, j=1,\dots,s\}$, \textit{i.e.},
\begin{equation}\label{gensamplingfmla}
f(t)=\sum\limits_{j=1}^{s}\sum\limits_{l\in\mathbb{Z}}(\Upsilon_jf)(\rho l)\Theta_{j}(t-\rho l),
\end{equation}
where the sampling period $\rho\in\mathbb{N}$ necessarily satisfies $\rho\leq s$ and $\Theta_j\in V(\phi),$ $j=1,\dots,s$ are the reconstruction functions. For generalized sampling expansion, we refer to \cite{AlSun, DjoVai, Garcia1, Garcia2, Garcia3, Papoulis, UnAl, UnZer1, UnZer2}. The authors in \cite {Garcia1} discussed the generalised sampling formulas for shift-invariant spaces, which include average sampling, shift sampling, and derivative sampling. They provided equivalent conditions for stable generalised sampling formulas in a shift-invariant space $V(\phi)$ under appropriate hypotheses on the generator $\phi$  and the involved systems. In this paper, we begin by examining their equivalent conditions using block Laurent operators (see Theorem \ref{equlntcondsspsicis}). We only focus on the time-invariant system of the form $\Upsilon_jf(t)=f^{(j)}(a+t)$, but the block Laurent operator technique can be easily applicable to other time-invariant systems.
In practice, it is difficult to verify the equivalent condition based on an invertible matrix $\Psi_{\kappa}(t)$ defined in \eqref{psi} even if the generator $\phi$ is compactly supported. In fact, the problem is equivalent to finding the zeros of polynomials on the unit circle.  

In the literature, the problem of uniform sampling with derivatives has been investigated for the space of splines of multiplicity $r$ \cite{LeSh, PlTa, Reimer, Schoenberg, ScSh}. Schoenberg \cite{Schoenberg} mentioned the Hermite-spline sampling problem on the equidistant lattices $\mathbb{Z}$ or $\tfrac{1}{2}+\mathbb{Z}$ for the space of splines of multiplicity $r$. Contrary to most of the papers addressing the aforementioned problem, we tend to discuss it for the shift-invariant spline spaces, which is the invertibility of the matrix $\Psi_{\kappa}(t)$ for $B$-spline generators.

To explain the problem of sampling with derivatives in $V(\phi)$ more precisely, let us first introduce the following terminology.

\begin{defn}
Let $\phi$ be a $(\rho-1)$-regular stable generator for $V(\phi)$. Then a set $\Gamma=\{x_n: n\in\mathbb{Z}\}$ of real numbers is said to be 
\begin{enumerate}
\item[$(i)$] a set of uniqueness (US) of order $\rho-1$ for $V(\phi)$ if $$ f\in V(\phi)~\text{and}~f^{(i)}(x_n)=0, n\in\mathbb{Z}, i= 0, 1,\dots, \rho-1,$$ together imply that $f\equiv0;$
\item[$(ii)$] a set of stable sampling (SS) of order $\rho-1$ for $V(\phi)$ if there exist constants $A, B>0$ such that $$A\|f\|^2_2\leq\sum\limits_{n\in\mathbb{Z}}\sum\limits_{i=0}^{\rho-1}|f^{(i)}(x_n)|^2\leq B\|f\|_2^2,$$
for all $f\in V(\phi);$
\item[$(iii)$] a set of interpolation (IS) of order $\rho-1$ for $V(\phi)$ if the interpolation problem $$f^{(i)}(x_n)=c_{ni}, n\in\mathbb{Z}, i= 0, 1, \dots, \rho-1,$$ has a solution $f\in V(\phi)$ for every square summable sequence $\{c_{ni}: n\in\mathbb{Z}, i= 0, 1, \dots, \rho-1\}.$
\end{enumerate}
If $\Gamma$ is both US and IS of order $\rho-1$ for $V(\phi),$  then it is called a complete interpolation set (CIS) of order $\rho-1$ for $V(\phi).$
\end{defn}
The B-splines are defined inductively as follows:
\begin{equation}\label{bspline}
Q_1(t)=\chi_{[0,1)}(t),~
Q_{m+1}(t)=\displaystyle{\int\limits_{-\infty}^{\infty}Q_m(t-y)Q_1(y)~dy},
~m\geq1.
\end{equation}
The spaces $V(Q_m),$ $m\geq1$ are called shift-invariant spline spaces.
 Schoenberg proved that $\mathbb{Z}$ (resp. $\tfrac{1}{2}+\mathbb{Z}$) is a CIS of order $0$ for shift-invariant even (resp. odd) spline spaces. It is known that $\rho\mathbb{Z} $ is a CIS of order $\rho-1$ for $\mathcal{B}_\pi$. So one may expect that $\rho\mathbb{Z}$ (resp. $\tfrac{1}{2}+\rho\mathbb{Z}$) is a CIS of order $\rho-1$ for shift-invariant even (resp. odd) spline spaces. In this paper, we demonstrate that this is not correct. Indeed, we need to choose a uniform sample set that depends not only on shift but also on the order of the splines and derivatives (see Section 3). Further, the sampling problem involving derivatives in $V(Q_m)$ introduces a new type of polynomials (see Tables 1 and 2). They are different from the Euler-Frobenius polynomials for the multiplicity $r>1$ that is mentioned in \cite{LeSh, Schoenberg}.
The chief aim of the paper is to show under what condition on $\kappa\equiv(Q_m,a,\rho),$ 
the sample set $a+\rho\mathbb{Z}$ is a CIS of order $\rho-1$ for $V(Q_m)$ from the invertibility of $\Psi_{\kappa}(t)$ (see Theorem \ref{cisoforderm-2} and Assumption \ref{assumption}).

There are numerous ways to construct approximation schemes with shift-invariant spaces, as Lei et al. pointed out in \cite{Lei}. 
They provided examples such as cardinal interpolation, quasi-interpolation, projection, and convolution (see also \cite{DDR, Jia, JiaLei}).
The authors in \cite {Garcia1, Garcia2, Garcia3} investigated the approximation from shift-invariant spaces by using generalized sampling formulas \eqref{gensamplingfmla}.
In \cite{Garcia1}, the authors discussed the approximation scheme based on average and shift sampling. In this paper, we construct an approximation scheme based on the derivative sampling formula in $V(\phi)$ as follows:
For a suitable real-valued function $f,$ consider the operator $S_W^{\kappa},$ formally defined as
\begin{equation}\label{derso}
(S_W^{\kappa}f)(t):= \sum\limits_{l\in\mathbb{Z}}\sum\limits_{i=0}^{\rho-1}\frac{1}{W^i}f^{(i)}\Big(\frac{a+\rho l}{W}\Big)\Theta_{i,\kappa}(Wt-\rho l),
\end{equation}
where $\Theta_{i,\kappa}$'s are called the reconstruction functions.
In \cite{Garcia2, Garcia3, KKS, Skopina}, the authors discussed the approximation scheme based on derivative sampling but they have chosen $\rho=1$ only. 

The second aim of the paper is to investigate the convergence and error analysis of the derivative sampling expansion \eqref{derso} for fundamentally large classes of functions.
If $f$ is continuous but not in $V(\phi)$, one normally investigates the convergence and error analysis of the sampling operator $S_W^{\kappa}$ as the dilation factor $W$ approaches infinity, together with aliasing error estimates, expressed in terms of the modulus of continuity of $f$ or its derivatives. 
In practice, however, signals are not always continuous and may contain jump discontinuities, such as shocks received by flying missiles, seismological shocks, or extrasystoles in the case of heartbeats. For this purpose, the authors \cite{BBSV1} discussed the rate of approximation of a signal $f$ by the following sampling operator
$$S_Wf(t)=\ds\sum_{n\in\mathbb{Z}}f\left(\frac{n}{W}\right)\text{sinc}(n-Wt)$$ 
based on the averaged modulus of smoothness $\tau_r(f;W^{-1})_p$ instead of  the classical modulus of continuity. Because the sinc function is not well suited to fast and efficient computation, it is rarely used in real-world applications.
Sampling in shift-invariant spaces generated by compactly supported functions is more accessible and realistic for many applications while retaining some of the simplicity and structure of bandlimited models. For this reason, authors  \cite{BBSV2} extended the results of \cite{BBSV1} for sampling series based on a suitable compactly supported kernel.
Inspired by the work in \cite{BBSV1, BBSV2} and the practical purpose of derivative sampling, we investigate the rate of approximation of a signal $f$ (not necessarily continuous) by the derivative sampling series $S_W^{\kappa}f$ in this paper.

The paper is organized as follows: Section 2 covers some basic facts related to derivative sampling,  the  space $\Lambda_{\rho}^p,$  $\tau$-modulus of continuity, and the interpolation theorem for linear operators from $\Lambda_{\rho}^p$ into $L^p(\mathbb{R})$. Section 3 is devoted to the investigation of the derivative sampling on the shift-invariant spline spaces. Section 4 is concerned with the rate of approximation of a signal $f$ by the sampling operator $S_W^{\kappa}f$ based on $\tau$-modulus. Section 5 contains examples to verify our results and three specific signals to show the error estimates. Finally, the Appendix gives the proof of the interpolation theorem. 
\section{Preliminary Results}
We first introduce the block Laurent operators in order to characterize the sampling and interpolation of uniform samples
involving derivatives.

For a function $f\in L^2[0,1]$, the Fourier transform
$\mathcal{F}:L^2(\mathbb{T})\to \ell^2(\mathbb{Z})$ is defined by $$(\mathcal{F}f)(n)=\widehat{f}(n)=\int_{0}^{1} f(t) e^{-2\pi\mathrm{i}n t}~dt,~n\in\mathbb{Z}.$$
In the sequel when there is no confusion possible, we use $\widehat{f}$ to denote either the Fourier coefficient or Fourier transform of a function.
A bounded linear operator $\mathcal{A}$  on $\ell^2(\mathbb{Z})$ associated with a matrix $\mathcal{A}=[a_{rs}]$ is said to be a \textit{Laurent operator} if $a_{r-k,s-k}=a_{r,s}$, for every $r$, $s$, $k$
$\in\mathbb{Z}$.
For $m\in L^\infty[0,1]$, let us define $\mathcal{M}:L^2[0,1]\to
L^2[0,1]$ by \[(\mathcal{M}f)(t):=m(t)f(t), ~\text{for}~a.e.~t\in [0,1].\]
Then the operator 
$\mathcal{A}=\mathcal{F}\mathcal{M}\mathcal{F}^{-1}$  is  called a Laurent operator defined by the symbol $m$. If it is invertible, then $\mathcal{A}^{-1}$ is also a Laurent operator with the symbol $\tfrac{1}{m}$ and the matrix of $\mathcal{A}^{-1}$ is given by
$$\mathcal{A}^{-1}=\left[\widehat{\left(\tfrac{1}{m}\right)}(r-s)\right].$$
Let $\mathcal{L}$ be an operator on $\ell^2(\mathbb{Z})^m$ associated with an $m\times m$ matrix whose entries are doubly infinite matrices, \textit{i.e.},
\begin{eqnarray*}
 \mathcal{L}=\left(
\begin{array}{ccccccc}
 L_{11} & L_{12}  & \cdots & L_{1m} \\
L_{21} & L_{22}  &  \cdots & L_{2m} \\
\vdots & \vdots  & \ddots & \vdots\\
L_{m1} & L_{m2}  & \cdots & L_{mm}  
 \end{array}
\right), ~ L_{rs} ~\text{is a doubly infinite matrix}.
\end{eqnarray*}
We say that $\mathcal{L}$ is a block Laurent operator on $\ell^2(\mathbb{Z})^m$  if all entries $L_{ij}$'s are Laurent operators on  $\ell^2(\mathbb{Z})$.
For an $m\times m$ matrix valued function $\Psi \in L_{m\times m}^1[0,1]$, the Fourier coefficient $\widehat{\Psi}$ of $\Psi$ is an $m\times m$  matrix 
$$ \widehat \Psi (k):= \int\limits_0^1  e^{-2\pi\mathrm{i} k  t}\Psi(t)~ dt, ~ k \in \mathbb{Z},$$
whose $(i,j)$-th entry is equal to the Fourier coefficients of the $(i, j)$-th entry of $\Psi.$
It is well-known that for every block Laurent operator $\mathcal{L}$, there exists a function $\Psi \in L^{\infty}_{m\times m}(\mathbb{T})$ such that $\mathcal{L}=[\widehat{\Psi}(s-j)]$ (see \cite{GoGoKa1}, p.565-566).
In this case, we say that $\mathcal{L}$ is a block Laurent operator defined by $\Psi$.
Since the class of Laurent operators is closed under addition, and multiplication and multiplication is commutative,
$$\det \mathcal{L}:=\sum\limits_\nu (sgn~\nu)L_{1\nu_1}\cdots L_{m\nu_m},$$
is a well-defined Laurent operator  on $\ell^2(\mathbb{Z})$ with the symbol $\det \Psi$. The following theorem is proved in \cite{GoGoKa1} (see also \cite{GhAn1}). 
\begin{thm}\label{IBL}
Let $\mathcal{L}=[L_{sj}]=[\widehat{\Psi}(s-j)]$ denote the matrix of a function $\Psi\in L^2_{m\times m}[0,1]$.Then
$\mathcal{L}$ is an invertible  block Laurent operator with the symbol $\Psi\in L^\infty_{m\times m}[0,1]$ if and only if there exist  two positive constants $A$ and $B$ such that $$A\leq |\det\Psi(t)|\leq B,~\text{for}~a.e. ~t\in [0,1].$$
\end{thm}
If $\phi$ is a $(\rho-1)$-regular stable generator for $V(\phi)$, then  
\begin{equation*}
f^{(i)}(t)= \sum\limits_{k\in\mathbb{Z}}\sum\limits_{j=0}^{\rho-1}c_{\rho k+j}\phi^{(i)}(t-\rho k-j),~0\leq i\leq\rho-1,
\end{equation*}
for every $f\in V(\phi)$. For each fixed $\kappa\equiv(\phi,a,\rho),$ we consider the sample set $\Gamma_{a,\rho}= a+\rho\mathbb{Z}$ for $V(\phi)$, where $\rho\in\mathbb{N}$ and $a\in[0, \rho).$  To each $\Gamma_{a,\rho},$ we associate the following infinite system
\begin{equation}\label{fp(x)}
f^{(i)}(a+\rho l)=\sum\limits_{k\in \mathbb{Z}}\sum\limits_{j=0}^{\rho-1} c_{\rho k+j}\phi^{(i)}(a+\rho(l-k)-j),
\end{equation}
for $l\in\mathbb{Z}$ and $i=0,1,\dots,\rho-1.$ Let us define the $\rho\times \rho$ block matrix
\begin{eqnarray}\label{entriesofu}
U_{\kappa}=
\begin{bmatrix}
U^{00}_{\kappa} & U^{01}_{\kappa} & \cdots& U^{0,\rho-1}_{\kappa}\\
U^{10}_{\kappa} & U^{11}_{\kappa} & \cdots& U^{1,\rho-1}_{\kappa}\\
\vdots & \vdots & \ddots & \vdots\\
U^{\rho-1,0}_{\kappa} & U^{\rho-1,1}_{\kappa} & \cdots& U^{\rho-1,\rho-1}_{\kappa}
\end{bmatrix},
\end{eqnarray} 
whose entries are given by $U^{ij}_{\kappa}=[\phi^{(i)}(a+\rho(l-k)-j)]_{l,k\in \mathbb{Z}}.$ Then we can write \eqref{fp(x)} in this form
\begin{equation}\label{uc=f}
U_{\kappa}C=F,
\end{equation}
where $C$ is a $\rho\times 1$ column vector 
\begin{eqnarray*}\label{C}
C=[\{c_{\rho k+0}\}^T_{k\in \mathbb{Z}},\{c_{\rho k+1}\}^T_{k\in \mathbb{Z}},\dots,\{c_{\rho k+\rho-1}\}^T_{k\in \mathbb{Z}}]^T
\end{eqnarray*}
and $F$ is a $\rho\times 1$ column vector
\begin{eqnarray*}\label{fn}
F=[\{f^{(0)}(a+\rho l)\}^T_{l\in \mathbb{Z}},\dots,\{f^{(\rho-1)}(a+\rho l)\}^T_{l\in \mathbb{Z}}]^T.
\end{eqnarray*}
Notice that $U^{ij}_{\kappa}$ is a Laurent operator with the symbol
\begin{equation}\label{laurentsymbol}
\Psi_{\kappa}^{ij}(t)= \sum\limits_{k\in\mathbb{Z}}\phi^{(i)}(a+\rho k-j)e^{2\pi \mathrm{i}kt}.
\end{equation}
Therefore $U_{\kappa}$ is a block Laurent operator defined by the symbol
\begin{equation}\label{psi}
\Psi_{\kappa}(t)=
\begin{bmatrix}
\Psi_{\kappa}^{00}(t) & \Psi_{\kappa}^{01}(t) & \cdots& \Psi_{\kappa}^{0,\rho-1}(t)\\
\Psi_{\kappa}^{10}(t) & \Psi_{\kappa}^{11}(t) & \cdots& \Psi_{\kappa}^{1,\rho-1}(t)\\
\vdots & \vdots & \ddots & \vdots\\
\Psi_{\kappa}^{\rho-1,0}(t) & \Psi_{\kappa}^{\rho-1,1}(t) & \cdots& \Psi_{\kappa}^{\rho-1,\rho-1}(t)
\end{bmatrix}.
\end{equation}
We can prove the following theorem using Theorem \ref{IBL} by making the same argument as in \cite{GhAn}.
\begin{thm}\label{equlntcondsspsicis}
Let $\phi$ be a $(\rho-1)$-regular stable generator for $V(\phi)$. Then the following statements are equivalent.
\begin{itemize}
\item [$(i)$] The sequence $\Gamma_{a,\rho}=\left\{a+\rho l: l\in\mathbb{Z}\right\}$ is an SS of order $\rho-1$ for $V(\phi)$.
\item[$(ii)$] $U_{\kappa}$ is an invertible block Laurent operator on $\ell^2(\mathbb{Z})^m.$
\item[$(iii)$] There exist $m, M>0$ such that the matrix $\Psi_{\kappa}(t)$ in \eqref{psi} satisfies that 
\begin{equation}\label{detpsi}
m\leq |\det \Psi_{\kappa}(t)|\leq M,~\text{for all}~t\in[0, 1].
\end{equation}
\item[($iv)$] $\Gamma_{a,\rho}$ is a CIS of order $\rho-1$ for $V(\phi)$.
\end{itemize}
In this case, every $f\in V(\phi)$ can be written as 
\begin{equation}\label{samplingfor}
f(t)=\sum\limits_{l\in\mathbb{Z}}\sum\limits_{i=0}^{\rho-1}f^{(i)}(a+\rho l)\Theta_{i,\kappa}(t-\rho l),
\end{equation}
where
\begin{equation}\label{theta}
\Theta_{i,\kappa}(t)= \sum_{v\in\mathbb{Z}}\sum\limits_{j=0}^{\rho -1}\widehat{\left(\Psi_{\kappa}^{-1}\right)^{ji}}(v)\phi(t-\rho v-j).
\end{equation}
\end{thm}

In order to estimate the error for the sampling series involving derivative samples, we extend the interpolation theorem from \cite{BBSV1}.
 Let $M(\mathbb{R})$ be the set of all bounded measurable functions on $\mathbb{R}$. Let $C(I)$ be the set of all continuous functions on the interval $I$ and $AC_{loc}^r(\mathbb{R})$ be the set of all $r$-fold locally absolutely continuous functions on $\mathbb{R}.$
Further, let 
\begin{equation*}
(\Delta_h^rf)(t):=\ds\sum_{j=0}^r(-1)^{r-j} {\binom{r}{j}} f(t+jh)
\end{equation*}
be the classical finite forward difference of order $r\in\mathbb{N}$ of $f$ with increment $h$ at the point $t$.
\begin{defn}
For $f\in M(\mathbb{R})$, the local modulus of smoothness of order $r\in\mathbb{N}$ at the point $x\in\mathbb{R} $  is defined for $\delta\geq 0$ by
$$\omega_r(f;x;\delta)
:= \sup\Big\{|(\Delta^r_hf)(t)|; t, t+rh\in\big[x-\tfrac{r\delta}{2}, x+\tfrac{r\delta}{2}\big] \Big\}.$$
\end{defn}
\begin{defn}
For $f\in M(\mathbb{R})$, $1\leq p<\infty,$ the $L^p$-averaged modulus of smoothness of order $r\in\mathbb{N}$ $($or $\tau$-modulus$)$ is defined for $\delta\geq0$ by
$$\tau_r(f;\delta)_p:=\|\omega_r(f;\bullet;\delta)\|_{L^p(\mathbb{R})}.$$
\end{defn}
\begin{defn}
\noindent
\begin{itemize}
\item[$(i)$]  A real sequence $\Sigma:= (x_j)_{j\in\mathbb{Z}}$ is said to be an admissible partition of $\mathbb{R}$ or an admissible sequence, if it satisfies 
$$0<\underline\Delta:=\inf\limits_{j\in\mathbb{Z}}(x_j-x_{j-1})\leq \overline\Delta:=\sup\limits_{j\in\mathbb{Z}}(x_j-x_{j-1})<\infty.$$
$\underline\Delta$ and $\overline\Delta$ are called the lower and upper mesh size, respectively.
\item[$(ii)$] Let $\Sigma:=(x_j)_{j\in\mathbb{Z}}$ be an admissible partition of $\mathbb{R},$ and let $\Delta_j= x_j-x_{j-1}.$ For a $(\rho-1)$-times differentiable function $f$ on $\Sigma,$ the discrete $\ell_{\rho}^p(\Sigma)$-norm of $f$ is defined for $1\leq p<\infty$ by
\begin{equation}\label{lpsigma}
\|f\|_{\ell_{\rho}^p(\Sigma)}:= \left\{\sum\limits_{j\in\mathbb{Z}}\sum\limits_{i=0}^{\rho-1}|f^{(i)}(x_j)|^p\Delta_j\right\}^{1/p}.
\end{equation}
\item[$(iii)$] The space $\Lambda_{\rho}^p$ for $1\leq p<\infty$ is defined by
$$\Lambda_{\rho}^p:=\left\{f\in M(\mathbb{R}): \|f\|_{\ell_{\rho}^p(\Sigma)}<\infty~ \text{for each admissible sequence} ~\Sigma\right\}.$$
\item[$(iv)$] The Sobolev spaces $W^r_p\equiv W^r(L^p(\mathbb{R})),$ $r\in\mathbb{N},$ $1\leq p<\infty,$ are given by
$$W_p^r:=\left\{f\in L^p(\mathbb{R}): f(t)=\phi(t) ~\textit{a.e.}, \phi\in AC^{(r)}_{loc}(\mathbb{R}), \phi^{(r)}\in L^p(\mathbb{R})\right\}.$$
\end{itemize}
\end{defn}

Arguing as in \cite {BBSV1}, we can show that if $f\in W_p^r\cap C(\mathbb{R})$ for some $r\geq\rho$, $1\leq p<\infty,$ then
\begin{equation}\label{inequality}
\|f\|_{\ell_{\rho}^p(\Sigma)}\leq\sum\limits_{i=0}^{\rho-1}\left(\|f^{(i)}\|_{L^p(\mathbb{R})}+ \overline\Delta\|f^{(i+1)}\|_{L^p(\mathbb{R})}\right),
\end{equation}
for any admissible partition $\Sigma$ with upper mesh size $\overline\Delta$ 
and hence $W_p^r\cap C(\mathbb{R})\subset\Lambda_{\rho}^p.$

Now we consider a family of linear operators $(L_{\beta})_{\beta\in I},$ $I$ being an index set, from $\Lambda_{\rho}^p$ to $L^p(\mathbb{R}).$ Further, $(\Sigma_{\beta})_{\beta\in I}$ is a family of admissible partitions $(x_{j,\beta})_{j\in\mathbb{Z}}$ with
$$\Delta_{j,\beta}= x_{j,\beta}-x_{j-1,\beta},~\overline\Delta_{\beta}=\sup\limits_{j\in\mathbb{Z}}(x_{j,\beta}-x_{j-1,\beta}), ~\text{and}~ \underline\Delta_{\beta}=\inf\limits_{j\in\mathbb{Z}}(x_{j,\beta}-x_{j-1,\beta}).$$
\begin{thm}\label{geninterpolation}
Let $(\Sigma_{\beta})_{\beta\in I}$ be a family of partitions as above with upper mesh sizes $\overline\Delta_{\beta}.$ Let $(L_{\beta})_{\beta\in I}$ be a family of linear operators mapping  $\Lambda_{\rho}^p$ into $L^p(\mathbb{R}),$ $1\leq p<\infty,$ satisfying the following properties:
\begin{itemize}
\item [$(i)$] $\|L_{\beta}f\|_{L^p(\mathbb{R})}\leq K_1\|f\|_{\ell_{\rho}^p(\Sigma_{\beta})}$, $(f\in\Lambda_{\rho}^p;~\beta\in I),$
\item [$(ii)$] $\|L_{\beta}g-g\|_{L^p(\mathbb{R})} \leq K_2\overline\Delta_{\beta}^s\|g^{(r)}\|_{L^p(\mathbb{R})},$ $(g\in W_p^r\cap C(\mathbb{R}), \beta\in I)$
\end{itemize}
for some fixed $r,s\in\mathbb{N}$ with $s\leq r$ and the constants $K_1,$ $K_2$ depending only on $\rho$ and $r.$ Then for each $f\in\Lambda_{\rho}^p$ and each $\overline\Delta_{\beta}\leq r$ there holds the estimate
\begin{equation}\label{interpolationthm}
\|L_{\beta}f-f\|_{L^p(\mathbb{R})}\leq \sum\limits_{i=0}^{\rho-1}c_i\tau_r(f^{(i)};\overline\Delta_{\beta}^{s/r})_p, ~(\beta\in I),
\end{equation}
where the constants $c_i,$ $i=0,1,\dots, \rho-1$  depend only on $K_1$ and $K_2.$ 
\end{thm}

The proof of the aforementioned theorem can be found in Appendix. Furthermore, one may generalise Theorem \ref{geninterpolation} to find the error estimate  of derivatives $f^{(s)}.$

\section{Derivative sampling in shift-invariant spline spaces}
We first recall some well-known facts in the theory of splines. The $B$-splines defined in \eqref{bspline} satisfy the following properties.
\begin{enumerate}
\item [$(i)$] $Q_m$ is compactly supported in the interval $[0,m]$.
 \item [$(ii)$] $Q_m(t)=\dfrac{1}{(m-1)!}\sum\limits_{j=0}^{m}(-1)^j
\binom{m}{j} \left(t-j\right)_+^{m-1}$, $m\geq 2$,
where $t_+=\max(0,t)$.
\item[$(iii)$] $Q_m^{(k)}(t)=\sum\limits_{r=0}^k(-1)^{r}{\binom{k}{r}}Q_{m-k}({t-r}),$ $k=0,1,\dots, m-2,~m\geq2.$
\item[$(iv)$] $Q_m$ forms a partition of unity, \textit{i.e.},     $\ds\sum_{k\in\mathbb{Z}}Q_m(t-k)=1$, for every $t\in\mathbb{R}$.
\item[$(v)$] $Q_m$ satisfies the Strang-Fix condition of order $m-1$, \textit{i.e.},
$$\widehat{Q_m}^{(r)}(l)=0,~r=0,1,\dots,m-1~\text{and}~ l\in\mathbb{Z}\setminus\{0\}.$$
\end{enumerate}
It is well-known that $f\in V(Q_m)$ if 
and only if $f\in C^{m-2}(\mathbb{R})\bigcap L^2(\mathbb{R}) $  and the restriction of $f$ to each interval $\left[n,n+1\right),n\in\mathbb{Z}$ is a polynomial of degree $\leq m-1$.
\begin{thm}\label{RBI}\cite{Mischenko, AntoRad1}.
The system $\{Q_m(\cdot-n):n\in\mathbb{Z}\}$ forms a Riesz basis for $V(Q_m)$ with optimal Riesz bounds $\dfrac{2^{2m-1}}{\pi^{2m-1}}K_{2m-1}$ and $1$, \textit{i.e.,}
\begin{equation}\label{RB}
\dfrac{2^{2m-1}}{\pi^{2m-1}}K_{2m-1}\sum_{n\in\mathbb{Z}}|d_n|^2\leq\big\|\sum_{n\in\mathbb{Z}}d_nQ_m(\cdot-n)\big\|^2_2\leq
\sum_{n\in\mathbb{Z}}|d_n|^2,
\end{equation}
for all $(d_n)\in\ell^2(\mathbb{Z})$,
where $K_m$'s are the Krein-Favard constants:
\begin{eqnarray*}
K_m=\dfrac{4}{\pi}\sum\limits_{\nu=0}^{\infty}\frac{(-1)^{\nu(m+1)}}{(2\nu+1)^{m+1}},~~~m=0,1,2,3,\dots.\\
\end{eqnarray*}
\end{thm}
\begin{thm}\label{cisoforderm-2}
Let $m\geq2.$ Then $(m-1)\mathbb{Z}$ is a CIS of order $m-2$ for the shift-invariant spline space $V(Q_m).$
\end{thm}

In order to prove Theorem \ref{cisoforderm-2}, we need the following lemmas.
\begin{lem}\cite{Ruiz}\label{ruiz}
For all integers $n\geq0$ and for all real numbers $t,$ we have
\begin{equation}
\sum\limits_{r=0}^{n}(-1)^r\binom{n}{r}(t-r)^l =
\begin{cases}
0, & ~\text{if}~l=0,1,\dots,n-1,\nonumber\\
n!, & ~\text{if}~ l=n.\nonumber
\end{cases}
\end{equation}

\end{lem}
\begin{lem}\label{binomial}
For all integers $n,l\geq0$ and for all $k\geq n,$ we have
\begin{equation*}
\sum\limits_{r=0}^{n}(-1)^r\binom{n}{r}\binom{k-r}{l}=
\begin{cases}
0, & ~\text{if}~l<n,\\
1,& ~\text{if}~l=n.
\end{cases}
\end{equation*}
Here and hereafter, it is understood that $\binom{\mu}{k}=0$ if $k>\mu.$
\end{lem}

\begin{pf}
For $n=0,$ the result is trivial. 
By the definition of the binomial coefficient, we can write 
$$\dbinom{k-r}{l}=\frac{1}{l!}\sum\limits_{s=1}^{l}a_{ls}(k-r)^s,~ a_{ls}\in\mathbb{R}, ~a_{ll}=1 .$$
Consequently, we have
\begin{eqnarray}\label{eqnlemma}
\sum\limits_{r=0}^{n}(-1)^r\binom{n}{r}\binom{k-r}{l}
=\frac{1}{l!}\sum\limits_{s=1}^{l}a_{ls}\sum\limits_{r=0}^{n}(-1)^r\binom{n}{r}(k-r)^s.
\end{eqnarray}
Now applying Lemma \ref{ruiz} in  \eqref{eqnlemma} for $n>0,$ we can conclude our result.
\end{pf}
\begin{lem}\label{elementaryoprtn}
For all $m\geq2$ and $0\leq i,l\leq m-2,$ we have
\begin{equation}
\sum\limits_{j=0}^{m-2}\binom{j}{l}\sum\limits_{r=0}^{i}(-1)^r\binom{i}{r}Q_{m-i}(m-1-j-r)=\begin{cases}
0, & ~\text{if}~l<i,\\
1, & ~\text{if}~l=i.
\end{cases}
\end{equation}
\end{lem}
\begin{pf}
We have 
\vspace{0.3cm}\\
\hspace*{0.5cm}$\displaystyle\sum\limits_{j=0}^{m-2}\binom{j}{l}\sum\limits_{r=0}^{i}(-1)^r\binom{i}{r}Q_{m-i}(m-1-j-r)$
\vspace{-0.3cm}
\begin{eqnarray}\label{lemma}
&=&\sum\limits_{r=0}^{i}(-1)^r\binom{i}{r}\sum\limits_{j=0}^{m-2}\binom{j}{l}Q_{m-i}(m-1-j-r)\nonumber\\
&=&\sum\limits_{r=0}^{i}(-1)^r\binom{i}{r}\sum\limits_{k=r}^{m-2+r}\binom{k-r}{l}Q_{m-i}(m-1-k)\nonumber\\
&=&\sum\limits_{k=i}^{m-2}Q_{m-i}(m-1-k)\sum\limits_{r=0}^{i}(-1)^r\binom{i}{r}\binom{k-r}{l},
\end{eqnarray}
using the fact that $Q_{m-i}(m-1-k)\neq0$ if and only if $i\leq k\leq m-2.$
Now applying Lemma \ref{binomial} in \eqref{lemma}, we get the desired result by using the partition of the unity property of the B-splines.
\end{pf}
\begin{pf}[\underline{\textbf{Proof of Theorem \ref{cisoforderm-2}}}]
From Theorem \ref{equlntcondsspsicis}, it is enough to show that the matrix $\Psi_{\kappa}(t)$  defined in \eqref{psi} satisfies \eqref{detpsi} for $a=0$ and $\rho=m-1$. In fact, we show  that $\det\Psi_{\kappa}(t)= e^{2\pi\mathrm{i}(m-1)t},$ for all $t\in[0,1].$
Now it follows from \eqref{laurentsymbol} and property $(iii)$ of B-splines that for $0\leq i,j\leq m-2,$
\begin{eqnarray}\label{matrixpsisym}
\Psi_{\kappa}^{ij}(t)
&=&\sum\limits_{k\in\mathbb{Z}}Q_m^{(i)}((m-1)k-j)e^{2\pi\mathrm{i}kt}\nonumber\\
&=&\sum\limits_{k\in\mathbb{Z}}\sum\limits_{r=0}^{i}(-1)^r\binom{i}{r}Q_{m-i}((m-1)k-j-r)e^{2\pi\mathrm{i}kt}.
\end{eqnarray}
Notice that $(m-1)k-j-r\geq 2(m-1)-j-i\geq m-i$ for $k>1.$
Since $Q_{m-i}(t)$ is compactly supported in $[0, m-i]$, it is now clear that $Q_{m-i}((m-1)k-j-r)=0~\text{for all}~ k\neq1.$
Therefore, we obtain from \eqref{matrixpsisym} that
$$\det\Psi_{\kappa}(t)=e^{2\pi\mathrm{i}(m-1)t}\det A,$$
where 
$$A=
\left(\begin{smallmatrix}
Q_m(m-1) & Q_m(m-2) & \dots & Q_m(1)\\
\sum\limits_{r=0}^{1}(-1)^r\tbinom{1}{r}Q_{m-1}(m-1-r) & \sum\limits_{r=0}^{1}(-1)^r\tbinom{1}{r}Q_{m-1}(m-2-r) & \dots & \sum\limits_{r=0}^{1}(-1)^r\tbinom{1}{r}Q_{m-1}(1-r)\\
\vdots & \vdots & \cdots & \vdots\\
\sum\limits_{r=0}^{m-2}(-1)^r\tbinom{m-2}{r}Q_{2}(m-1-r) & \sum\limits_{r=0}^{m-2}(-1)^r\tbinom{m-2}{r}Q_{2}(m-2-r) & \dots & \sum\limits_{r=0}^{m-2}(-1)^r\tbinom{m-2}{r}Q_{2}(1-r)\nonumber\\
\end{smallmatrix}\right).$$

We need to show that $\det A=1.$
Let $C_{\alpha}$ denote the $\alpha^{th}$ column of $A$ (for computational convenience, we take $\alpha$ from 0 to $m-2$). Replacing $C_{\alpha}$ by $\sum\limits_{j=0}^{m-2}\tbinom{j}{\alpha}C_{j},$  by Lemma \ref{elementaryoprtn}, $A$ is column equivalent to an upper triangular matrix $R$ with diagonal entries one. Observe that the elementary column operations we have performed on $A$ do not affect the determinant of $A.$
In fact, $R=AP,$ where
\begin{eqnarray*}
P=
\begin{bmatrix}
1 & 0 & \cdots & 0 & 0\\
1 & 1 & \cdots & 0 & 0\\
\vdots& \vdots & \ddots & \vdots & \vdots\\
\binom{m-3}{0} & \binom{m-3}{1} & \cdots & 1 & 0\\
\binom{m-2}{0} & \binom{m-2}{1} & \cdots & \binom{m-2}{m-3}& 1
\end{bmatrix}
\end{eqnarray*}
denote the lower triangular Pascal matrix of order $m-1$. Hence $\det A=1.$
\end{pf}

We now discuss the derivative sampling in $V(Q_m)$ for arbitrary values of $\kappa=(Q_m,a,\rho).$
It is easy to prove that if $\kappa=(Q_m,0,1)$, then $\det \Psi_{\kappa}(t)$ is a constant multiple of an Euler-Frobenious polynomial. Therefore 
$\mathbb{Z}$ is a CIS of order $0$ for $V(Q_m)$ if and only if $m$ is even.
Similarly, using the modified Euler-Frobenious polynomials we can conclude that
$\tfrac{1}{2}+\mathbb{Z}$ is a CIS of order $0$ for $V(Q_m)$ if and only if $m$ is odd. For further details, see \cite{Schoenberg, AntoRad}.

We have computed  
$\det \Psi_{\kappa}(t)$ for certain values  of $\kappa$ when $\rho=2$. 
In fact, for  $m=1,2,\dots,9$,
$$\det \Psi_{\kappa}(t)=
\begin{cases}
\tfrac{2^{m-2}}{(m-1)!(m-2)!}z^2P_{m-3}(z) ~~\text{if}~\kappa=(Q_m,0,2),\\\\
\tfrac{6}{(m-1)!(m-2)!2^{2m-3}}z\widetilde{P}_{m-2}(z)~~\text{if}~\kappa=(Q_m,\tfrac{1}{2},2),
\end{cases}$$
where  $z=e^{2\pi\mathrm{i}t}$ and see Table 1 (resp. Table 2 ) for $P_{m-3}$ (resp. $\widetilde{P}_{m-2}$).

\begin{center}
\begin{tabular}{|c|c|}
\hline
$m$ & $P_{m-3}(z)$\\
\hline
3 & $1$\\
\hline
4 & $1-z$\\
\hline
5 & $z^2-8z+1$\\
\hline
6 & $-z^3+39z^2-39z+1$\\
\hline
7 & $z^4-154z^3+666z^2-154z+1$\\
\hline
8 & $-z^5+545z^4-7750z^3+7750z^2-545z+1$\\
\hline
9 & $z^6-1812z^5+72759z^4-227576z^3+72759z^3-1812z^2+1$\\
\hline
\end{tabular}\\
\vspace{0.1cm}
Table $1.$ For $\kappa=(Q_m,0,2)$
\end{center}
\vspace{0.5cm}
\begin{center}
\begin{tabular}{|c|c|}
\hline
$m$ & $\widetilde{P}_{m-2}(z)$\\
\hline
3 & $1-z$\\
\hline
4 & $3z^2-38z+3$\\
\hline
5 & $9z^3-827z^2+827z-9$\\
\hline
6 & $27z^4-14636z^3+80418z^2-14636z+27$\\
\hline
7 & $81z^5-236885z^4+5082730z^3-5082730z^2+236885z-81$\\
\hline
8 & $243(z^6+1)-3681170(z^5+z)+257727933z^4-927852092z^3+257727933z^2$\\
\hline
9 & $729z^7-56136143z^6+11523750200z^5-141808460000z^4+141808460000z^3$\\
&$-11523750200z^2+56136143z-729$\\
\hline
\end{tabular}\\
\vspace{0.1cm}
Table $2.$ For $\kappa=(Q_m,\tfrac{1}{2},2)$
\end{center}
If a polynomial $p(z)=a_0+a_1z+\dots+a_{2n}z^{2n}$ satisfies $\sum_{i\neq n}|a_i|<|a_n|,$
then it has no zeros on the unit circle. 
Now Tables 1 and 2 show that we need to choose the uniform sample set that depends not only on shift but also on the order of the splines and derivatives. Furthermore, $P_{m-3}$ and $\widetilde{P}_{m-2}$ are different from the Euler-Frobenius polynomials for the multiplicity $r>1$ that is mentioned in \cite{LeSh, Schoenberg}.  
After computing  $\det \Psi_{\kappa}(t)$ for certain values of $\kappa$ with arbitrary $\rho$, we conjecture that the following assumption always holds. 

\begin{assumption}\label{assumption}
Let $\kappa=(Q_m,a,\rho)$ with $m>\rho$ and  $a\in \{0,\tfrac{1}{2}\}$. Then
 $|\det \Psi_{\kappa}(t)|\neq 0$ for all $t\in [0,1]$ if and only if  either  $m$ is even and $a=\left\langle\tfrac{\rho+1} {2}\right\rangle$ or $m$ is odd and $a=\left\langle\tfrac{\rho}{2}\right\rangle$. Here $\langle x\rangle$ denotes the fractional part of $x$.
\end{assumption}

Let $\lambda_{max}(A)$ (resp. $\lambda_{min}(A)$) denote the largest (resp. smallest) eigenvalue of a positive definite matrix $A$. The following result provides explicit sampling bounds, allowing one to use the frame algorithm or other algorithms to reconstruct a function from derivative samples.
\begin{thm}\label{upperlowerbound}
If Assumption $1$ holds, then 
$a+\rho\mathbb{Z}$ is a CIS of order $\rho-1$ for $V(Q_m)$.
Moreover, for every $f\in V(Q_m)$, we have
\begin{equation}\label{eqn3.9}
A_{\kappa}\|f\|_2^2\leq \sum_{i=0}^{\rho-1}\sum_{l\in \mathbb{Z}}|f^{(i)}(a+\rho l)|^2\leq B_{\kappa}\frac{\pi^{2m-1}}{2^{2m-1}K_{2m-1}}\|f\|_2^2,
\end{equation}
where $A_{\kappa}=\sup\limits_{t\in[0,1]}\lambda_{min}({\Psi^{*}_{\kappa}\Psi_{\kappa}}(t))$ and $B_{\kappa}=\sup\limits_{t\in[0,1]}\lambda_{max}(\Psi_{\kappa}^{*}\Psi_{\kappa}(t)).$
\end{thm}
\begin{proof}
Since $Q_m$ is compactly supported, $\det \Psi_{\kappa}(t)$ is continuous on $[0,1]$. Therefore the condition \eqref{detpsi} is equivalent to $|\det \Psi_{\kappa}(t)|\neq 0$ for all $t\in [0,1]$ and  hence $a+\rho\mathbb{Z}$ is a CIS of order $\rho-1$ for $V(Q_m)$.

If $f(t)=\sum\limits_{k\in\mathbb{Z}}c_{k}Q_{m}(t-k)$ for $(c_{k})\in l^{2}(\mathbb{Z})$, then it follows from Parseval's identity that
\begin{eqnarray}\label{eqn3.6}
 \sum_{i=0}^{\rho-1}\sum_{l\in \mathbb{Z}} |f^{(i)}(a+\rho l)|^2
&=& \sum_{i=0}^{\rho-1}\int\limits_{0}^{1}\Big|\sum\limits_{l\in\mathbb{Z}}\sum\limits_{k\in\mathbb{Z}}c_k Q^{(i)}_m(a+\rho l-k)e^{2\pi\mathrm{i}lt}\Big|^2~dt\nonumber\\
&=& \sum_{i=0}^{\rho-1}\int\limits_{0}^{1}\Big|\sum\limits_{j=0}^{\rho-1}\sum\limits_{n\in\mathbb{Z}}\sum\limits_{k\in\mathbb{Z}}c_{\rho k+j} Q^{(i)}_m(\rho n+a-j)
e^{2\pi\mathrm{i}(n+k)t}\Big|^2~dt\nonumber\\
&=& \sum_{i=0}^{\rho-1}\int\limits_{0}^{1}\left|\sum\limits_{j=0}^{\rho-1}\Psi^{ij}_{\kappa}(t)C_j(t)\right|^2~dt,
\end{eqnarray}
where $C_j(t)=\sum\limits_{k\in\mathbb{Z}}c_{\rho k+j}e^{2\pi\mathrm{i}kt}.$ If $C(t)=(C_0(t), \dots, C_{\rho-1}(t))^T$, we obtain from $\eqref{eqn3.6}$  and \eqref{RB} that
\begin{eqnarray}\label{eqn3.11}
\sum_{i=0}^{\rho-1}\sum_{l\in \mathbb{Z}} |f^{(i)}(a+\rho l)|^2
&=&\int\limits_{0}^{1}\|\Psi_{\kappa}(t)C(t)\|_{\mathbb{C}^\rho}^2~dt\\
&\leq&\int\limits_{0}^{1}\|\Psi_{\kappa}(t)\|^2\|C(t)\|_{\mathbb{C}^\rho}^2~dt\nonumber\\
&\leq&B_{\kappa}\int\limits_{0}^{1}\sum\limits_{j=0}^{\rho-1}\Big|\sum\limits_{k\in\mathbb{Z}}c_{\rho k+j}e^{2\pi\mathrm{i}kt}\Big|^2~dt\nonumber\\
&=&B_{\kappa}\sum\limits_{k\in\mathbb{Z}}|c_k|^2\nonumber\\
&\leq&B_{\kappa}\frac{\pi^{2m-1}}{2^{2m-1}K_{2m-1}}\|f\|^2_2.\nonumber
\end{eqnarray}
Since $\Psi_{\kappa}^{*}\Psi_{\kappa}(t)$ is a positive definite matrix,  there exists a positive definite matrix $\widetilde{\Psi}_{\kappa}(t)$ such that
$\widetilde{\Psi}^2_{\kappa}(t)=\Psi_{\kappa}^{*}\Psi_{\kappa}(t)$.
Therefore,
\begin{eqnarray*}
\hspace{-3cm}\|\Psi_{\kappa}(t)C(t)\|_{\mathbb{C}^\rho}^2
&=&|\langle\Psi_{\kappa}^{*}\Psi_{\kappa}(t)C(t), C(t)\rangle|\nonumber
\end{eqnarray*}
\begin{eqnarray}\label{lb}
&=&|\langle\widetilde{\Psi}_{\kappa}^2(t)C(t), C(t)\rangle|\nonumber\\
&=&|\langle\widetilde{\Psi}_{\kappa}(t)C(t), \widetilde{\Psi}_{\kappa}(t)C(t)\rangle|\nonumber\\
&=&\|\widetilde{\Psi}_{\kappa}(t)C(t)\|_{\mathbb{C}^\rho}^2\nonumber\\
&\geq&\tfrac{1}{\|\widetilde{\Psi}^{-1}_{\kappa}(t)\|^2}\|C(t)\|_{\mathbb{C}^\rho}^2.
\end{eqnarray}
Since
$\|\widetilde{\Psi}^{-1}_{\kappa}(t)\|
=\lambda_{max}(\widetilde{\Psi}^{-1}_{\kappa}(t))
=\tfrac{1}{\lambda_{min}(\widetilde{\Psi}_{\kappa}(t))}
=\tfrac{1}{\lambda_{min}(\sqrt{\Psi_{\kappa}^{*}\Psi_{\kappa}(t)})},$
our desired result follows from $\eqref{eqn3.11}$, $\eqref{lb}$, and $\eqref{RB}$.
\end{proof}

\section{Approximation by means of derivative sampling expansions}
Throughout this section, we assume that $\phi$ is a continuously $(\rho-1)$-times differentiable function with compact support. In this case, $\phi$ satisfies 
$$m_{\phi}:=\sup\limits_{u\in[0, 1]}\sum\limits_{k\in\mathbb{Z}}|\phi(u-k)|
<\infty.$$
Let $1\leq p<\infty.$ If $\phi$ is a stable generator for $V(\phi)$, then
$$V^p(\phi)= \Bigg\{f\in L^p(\mathbb{R}): f(t)= \sum\limits_{k\in\mathbb{Z}}c_k\phi(t-k), (c_k)\in \ell^p(\mathbb{Z})\Bigg\},$$
is a closed subspace of $L^p(\mathbb{R})$ and
there exist two positive constants $A_p,B_p$ such that
$$A_p\sum_{k\in\mathbb{Z}}|c_k|^p\leq \left\|\sum\limits_{k\in\mathbb{Z}}c_k\phi(\cdot-k)\right\|_{L^p(\mathbb{R})}^p\leq B_p\sum_{k\in\mathbb{Z}}|c_k|^p,$$
for every $(c_k)\in \ell^p(\mathbb{Z})$. Further, $\big\{\phi(\cdot-k):k\in\mathbb{Z} \big\}$ is an unconditional basis for $V^p(\phi)$ (see \cite{AlGr} and references therein).
Again arguing as in \cite{GhAn},  we can show that
$$f(t)=\sum\limits_{l\in\mathbb{Z}}\sum\limits_{i=0}^{\rho-1}f^{(i)}(a+\rho l)\Theta_{i,\kappa}(t-\rho l),$$ 
for every $f\in V^p(\phi)$, where $\Theta_{i,\kappa}$'s are defined in \eqref{theta}. We now consider the closed subspace 
$$V_W^p(\phi)= \Bigg\{f\in L^p(\mathbb{R}): f(t)= \sum\limits_{k\in\mathbb{Z}}c_k\phi(Wt-k), (c_k)\in \ell^p(\mathbb{Z})\Bigg\},~W\geq1.$$
By a change of variable, we can show that for every $f\in V_W^p(\phi),$ 
$$f(t)=\sum\limits_{l\in\mathbb{Z}}\sum\limits_{i=0}^{\rho-1}\frac{1}{W^i}f^{(i)}\left(\frac{a+\rho l}{W}\right)\Theta_{i,\kappa}(Wt-\rho l).$$
The above equation motivates the introduction of the following sampling operator $S_W^{\kappa},$ which is defined as follows:
\begin{equation}\label{so}
(S_W^{\kappa}f)(t):= \sum\limits_{l\in\mathbb{Z}}\sum\limits_{i=0}^{\rho-1}\frac{1}{W^i}f^{(i)}\Big(\frac{a+\rho l}{W}\Big)\Theta_{i,\kappa}(Wt-\rho l),
\end{equation}
where $f$ belongs to a suitable class of real or complex valued
functions.
The sampling operator $S_W^{\kappa}f$ is said to satisfy the reproducing polynomial property of order $r$ if $S_W^{\kappa}f(t)=f(t),$ for all polynomials of degree $\leq r.$

In order to estimate the error of the sampling operator $S_W^{\kappa},$ we need to deal with the uniform partitions $\Sigma_W:= \Big(\tfrac{a+\rho l}{W}\Big)_{l\in\mathbb{Z}}$ for $W\geq1.$
In this case, we write
$$\|f\|_{\ell_{\rho}^p(\Sigma_W)}\equiv\|f\|_{\ell_{\rho}^p(W)}=\left\{\sum\limits_{l\in\mathbb{Z}}\sum\limits_{i=0}^{\rho-1}\left|f^{(i)}\left(\frac{a+\rho l}{W}\right)\right|^p\frac{\rho}{W}\right\}^{1/p}.$$
\begin{lem}\label{lemmafortheta}
The following statements are equivalent.
\begin{itemize}
\item[$(i)$] The sampling operator $S_W^{\kappa}$ satisfies the reproducing polynomial property of order $r.$
\item[$(ii)$] For $n=0,1,\dots,r,$
\begin{equation}\label{thetacond}
\sum\limits_{i=0}^{\rho-1}\binom{n}{i}i!\sum\limits_{l\in\mathbb{Z}}\left(a+\rho l-t\right)^{n-i}\Theta_{i,\kappa}(t-\rho l)=\delta_{n0}.
\end{equation}
\item[$(iii)$] For $n=0,1,\dots,r$ and $l\in\mathbb{Z},$
\begin{equation}\label{ftthetacond}
\sum\limits_{i=0}^{\rho-1}\binom{n}{i}i!\sum\limits_{m=0}^{n-i}\binom{n-i}{m}a^m(2\pi\mathrm{i})^{m+i-n}\widehat{\Theta_{i,\kappa}}^{(n-i-m)}\left(\frac{l}{\rho}\right)=\rho\delta_{l0}\delta_{n0}.
\end{equation}
\end{itemize}
\end{lem}
\begin{pf}
By a change of variable, it is enough to prove the result for $W=1.$ 
Let
$$\displaystyle f_n(t):=\sum\limits_{i=0}^{\rho-1}\binom{n}{i}i!\sum\limits_{l\in\mathbb{Z}}\left(a+\rho l-t\right)^{n-i}\Theta_{i,\kappa}(t-\rho l).$$\\
Applying the binomial theorem, we get
\begin{eqnarray*}\label{derivativemoment2}
f_n(t)
&=&\sum\limits_{l\in\mathbb{Z}}\sum\limits_{i=0}^{\rho-1}\binom{n}{i}i!\sum\limits_{k=0}^{n-i}\binom{n-i}{k}\left(a+\rho l\right)^{k}(-t)^{n-i-k}\Theta_{i,\kappa}(t-\rho l)\\
&=&\sum\limits_{l\in\mathbb{Z}}\sum\limits_{i=0}^{\rho-1}\binom{n}{i}i!\sum\limits_{s=i}^{n}\binom{n-i}{s-i}(-t)^{n-s}\left(a+\rho l\right)^{s-i}\Theta_{i,\kappa}(t-\rho l)\\
&=&\sum\limits_{l\in\mathbb{Z}}\sum\limits_{i=0}^{\rho-1}i!\sum\limits_{s=i}^{n}\binom{n-i}{s-i}\binom{n}{i}(-t)^{n-s}\left(a+\rho l\right)^{s-i}\Theta_{i,\kappa}(t-\rho l)\\
&=&\sum\limits_{l\in\mathbb{Z}}\sum\limits_{i=0}^{\rho-1}i!\sum\limits_{s=0}^{n}\binom{n}{s}\binom{s}{i}(-t)^{n-s}\left(a+\rho l\right)^{s-i}\Theta_{i,\kappa}(t-\rho l)
\end{eqnarray*}
\begin{eqnarray*}
&=&\sum\limits_{s=0}^{n}\binom{n}{s}(-t)^{n-s}\sum\limits_{l\in\mathbb{Z}}\sum\limits_{i=0}^{\rho-1}\binom{s}{i}i!\left(a+\rho l\right)^{s-i}\Theta_{i,\kappa}(t-\rho l).
\end{eqnarray*}
Now it is easy to conclude that $(i)$ implies $(ii)$.
In a similar way, we can prove that $(ii)$ implies $(i)$.
Notice that $f_n(t)$ is a periodic function with period $\rho$. We can easily verify that 
\begin{eqnarray*}
\widehat{f_n}(l)
&:=&\frac{1}{\rho}\int\limits_{0}^{\rho}f_n(t)e^{-2\pi\mathrm{i}lt/\rho}~dt\\
&=&\frac{1}{\rho}\sum\limits_{i=0}^{\rho-1}\binom{n}{i}i!\sum\limits_{m=0}^{n-i}\binom{n-i}{m}a^m(2\pi\mathrm{i})^{m+i-n}\widehat{\Theta_{i,\kappa}}^{(n-i-m)}\left(\frac{l}{\rho}\right).
\end{eqnarray*}
Now the equivalence of $(ii)$ and $(iii)$ follows from the uniqueness of the Fourier series.
\end{pf}
\begin{prop}\label{cond1intpln}
For $1\leq p<\infty,$ there exists a constant $K_1>0$ independent of $W$ such that
\begin{equation}\label{inequalitySWf}
\|S_W^{\kappa}f\|_{L^p(\mathbb{R})}\leq K_1\|f\|_{\ell_{\rho}^p(W)},
\end{equation}
for every $f\in\Lambda_{\rho}^p$ and $W\geq1.$
\end{prop}
\begin{pf}
Applying H$\Ddot{o}$lder's inequality with $1/p+1/q= 1$ in  \eqref{so}, we have
\begin{eqnarray}\label{prop4.1}
&&\hspace{-1cm}|(S_W^{\kappa}f)(t)|\nonumber\\
&\leq&\sum\limits_{i=0}^{\rho-1}\frac{1}{W^i}\left\{\sum\limits_{l\in\mathbb{Z}}\left|\Theta_{i,\kappa}(Wt-\rho l)\right|\right\}^{\tfrac{1}{q}}\left\{\sum\limits_{l\in\mathbb{Z}}\left|f^{(i)}\left(\frac{a+\rho l}{W}\right)\right|^p\left|\Theta_{i,\kappa}(Wt-\rho l)\right|\right\}^{\tfrac{1}{p}}\nonumber\\
&\leq&\sum\limits_{i=0}^{\rho-1}\frac{1}{W^i}\left\{\sum\limits_{l\in\mathbb{Z}}\sum\limits_{v\in\mathbb{Z}}\sum\limits_{j=0}^{\rho-1}\left|\widehat{(\Psi_{\kappa}^{-1})^{ji}}(v)\right|\left|\phi(Wt-\rho (l+ v)-j)\right|\right\}^{1/q}\nonumber\\
&&\hspace{4.5cm}\times\left\{\sum\limits_{l\in\mathbb{Z}}\left|f^{(i)}\left(\frac{a+\rho l}{W}\right)\right|^p\left|\Theta_{i,\kappa}(Wt-\rho l)\right|\right\}^{1/p}\nonumber\\
&=&\sum\limits_{i=0}^{\rho-1}\frac{1}{W^i}\left\{\sum\limits_{j=0}^{\rho-1}\sum\limits_{v\in\mathbb{Z}}\left|\widehat{(\Psi_{\kappa}^{-1})^{ji}}(v)\right|\sum\limits_{k\in\mathbb{Z}}\left|\phi(Wt-\rho k-j)\right|\right\}^{1/q}\nonumber\\
&&\hspace{4.5cm}\times\left\{\sum\limits_{l\in\mathbb{Z}}\left|f^{(i)}\left(\frac{a+\rho l}{W}\right)\right|^p\left|\Theta_{i,\kappa}(Wt-\rho l)\right|\right\}^{1/p}\nonumber\\
&\leq&m_{\phi}^{1/q}\sum\limits_{i=0}^{\rho-1}\frac{M_i^{1/q}}{W^i}\left\{\sum\limits_{l\in\mathbb{Z}}\left|f^{(i)}\left(\frac{a+\rho l}{W}\right)\right|^p\left|\Theta_{i,\kappa}(Wt-\rho l)\right|\right\}^{1/p},
\end{eqnarray}
where $M_i=\sum\limits_{j=0}^{\rho-1}\sum\limits_{v\in\mathbb{Z}}\left|\widehat{(\Psi_{\kappa}^{-1})^{ji}}(v)\right|.$

Since $\Psi_{\kappa}^{ij}(x)$ is infinitely differentiable on $[0, 1]$,  $(\Psi_{\kappa}^{-1})^{ji}$ is also infinitely differentiable. Consequently, for any $r>1,$ $$\left|\widehat{(\Psi_{\kappa}^{-1})^{ji}}(v)\right|\leq \frac{c}{|v|^r},~\text{for each}~i, j=0,1,\dots, \rho-1.$$ 
Hence $M:= \max(M_0, M_1,\dots, M_{\rho-1})<\infty.$ Therefore, \eqref{prop4.1} becomes
\begin{eqnarray}\label{prop4swf}
\left|(S_W^{\kappa}f)(t)\right|
\leq M^{1/q}m_{\phi}^{1/q}\sum\limits_{i=0}^{\rho-1}\frac{1}{W^i}\left\{\sum\limits_{l\in\mathbb{Z}}\left|f^{(i)}\left(\frac{a+\rho l}{W}\right)\right|^p\left|\Theta_{i,\kappa}(Wt-\rho l)\right|\right\}^{1/p}.
\end{eqnarray}
Applying the inequality $(a+b)^p\leq2^p(a^p+b^p)$ repeatedly, \eqref{prop4swf} yields
\begin{eqnarray*}
&&\hspace{-1cm}\int\limits_{-\infty}^{\infty}\left|(S_W^{\kappa}f)(t)\right|^p dt\\
&\leq&M^{p/q}m_{\phi}^{p/q}\int\limits_{-\infty}^{\infty}\left[\sum\limits_{i=0}^{\rho-1}\frac{1}{W^i}\left\{\sum\limits_{l\in\mathbb{Z}}\left|f^{(i)}\left(\frac{a+\rho l}{W}\right)\right|^p\left|\Theta_{i,\kappa}(Wt-\rho l)\right|\right\}^{1/p}\right]^p dt\\
&\leq&2^{(\rho-1) p}M^{p/q}m_{\phi}^{p/q}\sum\limits_{i=0}^{\rho-1}\frac{1}{W^{ip}}\int\limits_{-\infty}^{\infty}\sum\limits_{l\in\mathbb{Z}}\left|f^{(i)}\left(\frac{a+\rho l}{W}\right)\right|^p\left|\Theta_{i,\kappa}(Wt-\rho l)\right| dt\\
&\leq&2^{(\rho-1) p}M^{p/q}m_{\phi}^{p/q}\sum\limits_{i=0}^{\rho-1}\frac{1}{W^{ip}}\sum\limits_{l\in\mathbb{Z}}\left|f^{(i)}\left(\frac{a+\rho l}{W}\right)\right|^p\\
&&\hspace{3.5cm}\times\sum\limits_{j=0}^{\rho-1}\sum\limits_{v\in\mathbb{Z}}\left|\widehat{(\Psi_{\kappa}^{-1})^{ji}}(v)\right|\int\limits_{-\infty}^{\infty}\left|\phi(Wt-\rho (l+v)-j)\right| dt\\
&=&\frac{2^{(\rho-1) p}}{W}M^{p/q}m_{\phi}^{p/q}\|\phi\|_{L^1(\mathbb{R})}\sum\limits_{i=0}^{\rho-1}\sum\limits_{j=0}^{\rho-1}\sum\limits_{v\in\mathbb{Z}}\left|\widehat{(\Psi_{\kappa}^{-1})^{ji}}(v)\right|\sum\limits_{l\in\mathbb{Z}}\left|\frac{1}{W^i}f^{(i)}\left(\frac{a+\rho l}{W}\right)\right|^p \\
&\leq&2^{(\rho-1) p}M^{\frac{p}{q}+1}m_{\phi}^{p/q}\|\phi\|_{L^1(\mathbb{R})}\sum\limits_{l\in\mathbb{Z}}\sum\limits_{i=0}^{\rho-1}\left|f^{(i)}\left(\frac{a+\rho l}{W}\right)\right|^p\frac{\rho}{W},\\
\end{eqnarray*}
from which we obtain the desired inequality \eqref{inequalitySWf}.
\end{pf}

\begin{prop}
Let $r\geq\rho$. Suppose that $S_W^{\kappa}f$ satisfies the reproducing polynomial property of order $r.$
If $g\in W_p^r\cap C(\mathbb{R}),$ then there exists a constant $K>0$ independent of $W$ such that 
\begin{equation}\label{proposition4.2}
\|S_W^{\kappa}g-g\|_{L^p(\mathbb{R})}\leq KW^{-r}\|g^{(r)}\|_{L^p(\mathbb{R})},~(W\geq1).
\end{equation}
\end{prop}
\begin{pf}
Applying Taylor's formula with the integral remainder to the functions $g^{(i)},$ for each $i=0,1,\dots, \rho-1,$ we obtain
\begin{multline}\label{taylorsfrmla}
g^{(i)}\left(\frac{a+\rho l}{W}\right)
=\sum\limits_{s=i}^{r-1}\frac{g^{(s)}(t)}{(s-i)!}\left(\frac{a+\rho l}{W}-t\right)^{s-i}\nonumber\\
+ \frac{1}{(r-1-i)!}\int\limits_{t}^{\frac{a+\rho l}{W}}g^{(r)}(u)\left(\frac{a+\rho l}{W}-u\right)^{r-1-i} du,  
\end{multline}
for $l\in\mathbb{Z}$ and $t\in\mathbb{R}.$
It follows that
\begin{eqnarray*}
&&\hspace{-1cm}(S_W^{\kappa}g)(t)\\
&=&\sum\limits_{l\in\mathbb{Z}}\sum\limits_{i=0}^{\rho-1}\frac{1}{W^i}\sum\limits_{s=i}^{r-1}\frac{g^{(s)}(t)}{(s-i)!}\left(\frac{a+\rho l}{W}-t\right)^{s-i}\Theta_{i,\kappa}(Wt-\rho l)\\
&&+\sum\limits_{i=0}^{\rho-1}\frac{1}{W^i(r-1-i)!}\sum\limits_{l\in\mathbb{Z}}\Theta_{i,\kappa}(Wt-\rho l)\int\limits_{t}^{\frac{a+\rho l}{W}}g^{(r)}(u)\left(\frac{a+\rho l}{W}-u\right)^{r-1-i} ~du\\
&=&\sum\limits_{s=0}^{r-1}\frac{g^{(s)}(t)}{s!}\sum\limits_{l\in\mathbb{Z}}\sum\limits_{i=0}^{\rho-1}\frac{i!}{W^i}\binom{s}{i}\left(\frac{a+\rho l}{W}-t\right)^{s-i}\Theta_{i,\kappa}(Wt-\rho l)\\
&&+\sum\limits_{i=0}^{\rho-1}\frac{1}{W^i(r-1-i)!}\sum\limits_{l\in\mathbb{Z}}\Theta_{i,\kappa}(Wt-\rho l)\int\limits_{t}^{\frac{a+\rho l}{W}}g^{(r)}(u)\left(\frac{a+\rho l}{W}-u\right)^{r-1-i} ~du.\\
\end{eqnarray*}
Applying Lemma \ref{lemmafortheta}, we get
\begin{eqnarray}
&&\hspace{-1cm}S_W^{\kappa}g(t)-g(t)\nonumber\\
&=&\sum\limits_{i=0}^{\rho-1}\frac{1}{W^i(r-1-i)!}\sum\limits_{l\in\mathbb{Z}}\Theta_{i,\kappa}(Wt-\rho l)\int\limits_{t}^{\frac{a+\rho l}{W}}g^{(r)}(u)\left(\frac{a+\rho l}{W}-u\right)^{r-1-i} ~du\label{SWg-g}\nonumber\\
&=&\sum\limits_{i=0}^{\rho-1}\frac{1}{W^{r-1}(r-1-i)!}\sum\limits_{l\in\mathbb{Z}}\Theta_{i,\kappa}(Wt-\rho l)\int\limits_{t}^{\frac{a+\rho l}{W}}g^{(r)}(u)\left(a+\rho l-Wu\right)^{r-1-i}~du\nonumber\\
&=:&S_1(t)+S_2(t)\label{s1+s2},
\end{eqnarray}
where
\begin{multline*}
S_1(t)=\frac{1}{W^{r-1}}\sum\limits_{v\geq0}\sum\limits_{l\in\mathbb{Z}}\sum\limits_{i=0}^{\rho-1}\frac{1}{(r-1-i)!}\sum\limits_{j=0}^{\rho-1}\widehat{(\Psi_{\kappa}^{-1})^{ji}}(v)\phi(Wt-\rho l-\rho v-j)\\
\times\int\limits_{t}^{\tfrac{a+\rho l}{W}}g^{(r)}(u)(a+\rho l-Wu)^{r-1-i}~du
\end{multline*}
and $S_2(t)$ is defined in a similar manner as $S_1(t),$ but $v\geq0$ is replaced by $v<0.$

Since $\phi$ is compactly supported, there exist $T_0, T_1\in\mathbb{R}$ such that $\phi(t)=0$ for all $t\notin[T_0, T_1].$ Let $T=\max\{|T_0|, |T_1|\}.$ If
$T_0\leq Wt-\rho l-\rho v-j\leq T_1$ and $v\geq0,$ then $-\rho v-j-T\leq Wt-\rho l\leq T+\rho v+j,$ which implies that $\left|Wt-\rho l\right|\leq T+\rho v+\rho-1.$

Let $A_{v,t,\rho}:=\{l\in\mathbb{Z}:|Wt-\rho l|\leq T+\rho v+\rho-1\}.$ Then
\begin{eqnarray*}
|S_1(t)|&\leq&\frac{1}{W^{r-1}}\sum\limits_{v\geq0}\sum\limits_{l\in A_{v,t,\rho}}\sum\limits_{i=0}^{\rho-1}\frac{1}{(r-1-i)!}\sum\limits_{j=0}^{\rho-1}\left|\widehat{(\Psi_{\kappa}^{-1})^{ji}}(v)\right|\left|\phi(Wt-\rho l-\rho v-j)\right|\\
&&\hspace{2cm}\times\int\limits_{t}^{\tfrac{a+\rho l}{W}}\left|g^{(r)}(u)\right| \left|a+\rho l-Wu\right|^{r-1-i}~du.
\end{eqnarray*}
When $u\in \left[t, \frac{a+\rho l}{W}\right]$ and $|Wt-\rho l|\leq T+\rho v+\rho-1,$ we have 
$$|a+\rho l-Wu|\leq a+T+\rho-1+\rho v.$$
Hence
\begin{eqnarray*}
|S_1(t)|
&\leq&\frac{1}{W^{r-1}}\sum\limits_{v\geq0}\sum\limits_{l\in A_{v,t,\rho}}\sum\limits_{i=0}^{\rho-1}\frac{(a+T+\rho-1+\rho v)^{r-1-i}}{(r-1-i)!}\\
&&\hspace{1.5cm}\times\sum\limits_{j=0}^{\rho-1}\left|\widehat{(\Psi_{\kappa}^{-1})^{ji}}(v)\right|\left|\phi(Wt-\rho l-\rho v-j)\right|
\int\limits_{0}^{\tfrac{a+\rho l}{W}-t}|g^{(r)}(u+t)| du \\
&\leq&\frac{m_{\phi}}{W^{r-1}}\sum\limits_{v\geq0} \sum\limits_{i=0}^{\rho-1}\frac{(a+T+\rho-1+\rho v)^{r-1-i}}{(r-1-i)!}\sum\limits_{j=0}^{\rho-1}\left|\widehat{(\Psi_{\kappa}^{-1})^{ji}}(v)\right|\\
&&\hspace{1.5cm}\times\int\limits_{|u|\leq\tfrac{a+T+\rho-1+\rho v}{W}}|g^{(r)}(u+t)| du.\\
\end{eqnarray*}
By Minkowski's integral inequality, we get
\begin{eqnarray}\label{s1}
\|S_1\|_{L^p(\mathbb{R})}
&\leq&\frac{m_{\phi}}{W^{r-1}}\sum\limits_{v\geq0}\sum\limits_{i=0}^{\rho-1}\frac{(a+T+\rho-1+\rho v)^{r-1-i}}{(r-1-i)!}\sum\limits_{j=0}^{\rho-1}|\widehat{(\Psi_{\kappa}^{-1})^{ji}}(v)|\nonumber\\
&&\hspace{1.5cm}\times\int\limits_{|u|\leq\tfrac{a+T+\rho-1+\rho v}{W}}\|g^{(r)}(u+\bullet)\|_{L^p(\mathbb{R})} du\nonumber\\
&=&\frac{2m_{\phi}}{W^{r}}\|g^{(r)}\|_{L^p(\mathbb{R})}\nonumber\sum\limits_{v\geq0}\sum\limits_{i=0}^{\rho-1}\frac{(a+T+\rho-1+\rho v)^{r-i}}{(r-1-i)!}\sum\limits_{j=0}^{\rho-1}|\widehat{(\Psi_{\kappa}^{-1})^{ji}}(v)|\nonumber\\
&=&\frac{2m_{\phi}}{W^{r}(r-1)!}h_1\|g^{(r)}\|_{L^p(\mathbb{R})},
\end{eqnarray}
where $h_1
=\sum\limits_{v\geq0}\sum\limits_{i=0}^{\rho-1}\binom{r-1}{i}i!(a+T+\rho-1+\rho v)^{r-i}\sum\limits_{j=0}^{\rho-1}\left|\widehat{(\Psi_{\kappa}^{-1})^{ji}}(v)\right|.$
Similarly, we can prove that
\begin{eqnarray}\label{s2}
\|S_2\|_{L^p(\mathbb{R})}
&\leq&\frac{2m_{\phi}}{W^{r}(r-1)!}h_2\|g^{(r)}\|_{L^p(\mathbb{R})},
\end{eqnarray}
where $h_2
=\sum\limits_{v<0}\sum\limits_{i=0}^{\rho-1}\binom{r-1}{i}i!(a+T+\rho-1+\rho v)^{r-i}\sum\limits_{j=0}^{\rho-1}\left|\widehat{(\Psi_{\kappa}^{-1})^{ji}}(-v)\right|.$
Since $(\Psi_{\kappa}^{-1})^{ji}$ is infinitely differentiable on $[0, 1],$ it is easy to show that $h_1$ and $h_2$  are finite.
Consequently, we obtain the desired inequality \eqref{proposition4.2} from \eqref{s1+s2}, \eqref{s1}, and \eqref{s2}.
\end{pf}

Since $S_W^{\kappa}$ satisfies $(i)$ and $(ii)$ of Theorem \ref{geninterpolation} for uniform partitions $\Sigma_W$, we obtain the following
\begin{thm}\label{smplngoperatorthm}
Let $r\geq\rho.$ Assume that the sampling operator  $S_W^{\kappa}$ satisfies the reproducing polynomial property of order $r.$ Then for $1\leq p<\infty,$ 
\begin{equation}
\|S_W^{\kappa}f-f\|_{L^p(\mathbb{R})}\leq \sum\limits_{i=0}^{\rho-1}c_i\tau_r(f^{(i)}, W^{-1})_p,~ (f\in\Lambda_{\rho}^p;~W\geq1),
\end{equation}
where $c_i,$ $i=0,1,\dots,\rho-1$ depend on $\rho$ and $r$ only.
\end{thm}

\begin{cor}\label{approxcor}
Let $1\leq p<\infty$ and $\alpha>0.$  Assume that the sampling operator $S_W^{\kappa}$  satisfies the reproducing polynomial property of order $r.$ If $f\in\Lambda_{\rho}^p$ such that  $\tau_r(f^{(i)},\delta)_p=\mathcal{O}(\delta^{\alpha})$ as $\delta\to0+$ for each $i=0,1,\dots,\rho-1,$ then
\begin{equation}
\|S_W^{\kappa}f-f\|_{L^p(\mathbb{R})}= \mathcal{O}(W^{-\alpha})~~(W\to\infty).
\end{equation}
\end{cor}

\section{examples}
We now illustrate our results for shift-invariant spline space $V(Q_m)$ for certain values of $\kappa.$
Taking Fourier transform on \eqref{theta}, we obtain that
\begin{eqnarray}\label{thetaroof}
\widehat{\Theta_{i,\kappa}}(\xi)
&=&\sum_{v\in\mathbb{Z}}\sum\limits_{j=0}^{\rho -1}\widehat{\left(\Psi_{\kappa}^{-1}\right)^{ji}}(v)\widehat{\phi}(\xi)e^{-2\pi\mathrm{i}(\rho v+j)\xi}\nonumber\\
&=&\widehat{\phi}(\xi)\sum\limits_{j=0}^{\rho-1}\left(\Psi_{\kappa}^{-1}\right)^{ji}(-\rho\xi)e^{-2\pi\mathrm{i}j\xi}.
\end{eqnarray}
Throughout this section, we let $z=e^{2\pi\mathrm{i}\xi}.$
\subsection{For $\kappa=(Q_3,0,2)$}
By Theorem \ref{cisoforderm-2}, $2\mathbb{Z}$ is a CIS  of order $1$ for $V(Q_3)$. In this case  
$$\Psi_{\kappa}(t)=e^{2\pi\mathrm{i}t}
\begin{bmatrix} 
\frac{1}{2} & \frac{1}{2}\\\\
-1 & 1
\end{bmatrix}
~\text{and}~
\Psi^{-1}_{\kappa}(t)=
e^{-2\pi\mathrm{i}t}\begin{bmatrix}
1 & -\frac{1}{2}\\ \\
1 & \frac{1}{2}
\end{bmatrix}. 
$$
The Fourier coefficients  of $\Psi^{-1}_{\kappa}$ are given by
$$\widehat{\Psi^{-1}_{\kappa}}(-1)= \begin{bmatrix}
1 & -\tfrac{1}{2} \\ \\
1 &  \tfrac{1}{2}
\end{bmatrix}~\text{and}~
\widehat{\Psi^{-1}_{\kappa}}(n)=
\begin{bmatrix}
0&0\\ 0&0
\end{bmatrix}~\text{for}~ n\neq-1.$$
Consequently,
$$\Theta_{0,{\kappa}}(t)=Q_3(t+2)+Q_3(t+1)
~\text{and}~ \Theta_{1,\kappa}(t)
=-\frac{1}{2}Q_3(t+2)+\frac{1}{2}Q_3(t+1).$$
Hence we obtain the following sampling formula 
$$
f(t)=\sum\limits_{l\in\mathbb{Z}}\left\{\left[f(2l)-\tfrac{1}{2}f^{\prime}(2l)\right]Q_3(t-2l+2)+\left[f(2l)+\tfrac{1}{2}f^{\prime}(2l)\right]Q_3(t-2l+1)\right\}.
$$
Since
$$\Psi_{\kappa}^{*}\Psi_{\kappa}(t)
=\begin{bmatrix}
\frac{5}{4} & -\frac{3}{4}  \\ \\
-\frac{3}{4} & \frac{5}{4}
\end{bmatrix},~
\lambda_{max}(\Psi^*_{\kappa}\Psi_{\kappa}(t))= 2,~\text{ and}~ \lambda_{min}(\Psi^*_{\kappa}\Psi_{\kappa}(t))= \frac{1}{2},$$
we obtain that
\begin{equation}
\frac{1}{2}\|f\|^2_2\leq\sum_{l\in \mathbb{Z}}\sum_{i=0}^{1}|f^{(i)}(2l)|^2\leq15\|f\|^2_2,\nonumber
\end{equation}
from Theorem \ref{upperlowerbound}. We now obtain from \eqref{thetaroof} that
$$\widehat{\Theta_{0,{\kappa}}}(\xi)=\widehat{Q_3}(\xi)\left(z^2+z\right)~\text{and}~
\widehat{\Theta_{1,\kappa}}(\xi)=\tfrac{1}{2}\widehat{Q_3}(\xi)\left(-z^2+z\right).$$
So
\begin{equation*}
\hspace{-4.65cm}\widehat{\Theta_{0,\kappa}}^{\prime}({\xi})=2\pi\mathrm{i}\widehat{Q_3}(\xi)\left(2z^2+z\right)+\widehat{Q_3}^{\prime}(\xi)\left(z^2+z\right),
\end{equation*}
\begin{equation*}
\widehat{\Theta_{0,\kappa}}^{\prime\prime}({\xi})=(2\pi\mathrm{i})^2\widehat{Q_3}(\xi)\left(4z^2+z\right)+4\pi\mathrm{i}\widehat{Q_3}^{\prime}(\xi)\left(2z^2+z\right)+\widehat{Q_3}^{\prime\prime}(\xi)\left(z^2+z\right),
\end{equation*}
and
\begin{equation*}
\hspace{-4cm}\widehat{\Theta_{1,\kappa}}^{\prime}(\xi)=\pi\mathrm{i}\widehat{Q_3}(\xi)\left(-2z^2+z\right)+\tfrac{1}{2}\widehat{Q_3}^{\prime}(\xi)\left(-z^2+z\right).
\end{equation*}
It is easy to verify that $\widehat{Q_3}(0)=1,$ $\widehat{Q_3}^{\prime}(0)=-3\pi\mathrm{i},$ and $\widehat{Q_3}^{\prime\prime}(0)=-10\pi^2.$ By Strang-Fix condition of $Q_3,$ we have
\begin{equation*}
\hspace{-6.9cm}\widehat{\Theta_{0,\kappa}}(l/2)=
\begin{cases}
2 &~\text{if}~l=0,\\
0 &~\text{if}~l\in\mathbb{Z}\setminus\{0\},
\end{cases}
\end{equation*}
\begin{equation*}
\hspace{-2.9cm}\widehat{\Theta_{1,\kappa}}(l/2)=\frac{\mathrm{i}}{2\pi}\widehat{\Theta_{0,\kappa}}^{\prime}(l/2)=
\begin{cases}
0 &~\text{if}~l\in2\mathbb{Z},\\
-\widehat{Q_3}\left(\tfrac{l}{2}\right)&~\text{if}~l\in1+2\mathbb{Z},
\end{cases}
\end{equation*}
\begin{equation*}
\widehat{\Theta_{1,\kappa}}^{\prime}(l/2)
=\frac{\mathrm{i}}{4\pi}\widehat{\Theta_{0,\kappa}}^{\prime\prime}(l/2)=
\begin{cases}
 -\pi\mathrm{i} &~\text{if}~l=0,\\
 0 &~\text{if}~l\in2\mathbb{Z}\setminus\{0\},\\
-3\pi\mathrm{i}\widehat{Q_3}\left(\tfrac{l}{2}\right)-\widehat{Q_3}^{\prime}\left(\tfrac{l}{2}\right) &~\text{if}~l\in1+2\mathbb{Z}.
\end{cases}\nonumber
\end{equation*}
By substituting the values of $\widehat{\Theta_{i,\kappa}}^{(r)}(l/2)$ in \eqref{ftthetacond}, we can show that the operator $S_W^{\kappa}$ satisfies the reproducing polynomial property of order $2.$
Hence we obtain the following result from Corollary \ref{approxcor}.
\begin{cor}
Let $\kappa=(Q_3,0,2).$ If $f\in\Lambda_{2}^p$ such that 
$\tau_2(f^{(i)},\delta)_p=\mathcal{O}(\delta^{\alpha})$ as $\delta\to0+$ for each $i=0,1,$  then
\begin{equation}
\|S_W^{\kappa}f-f\|_{L^p(\mathbb{R})}= \mathcal{O}(W^{-\alpha})~ (W\to\infty).
\end{equation}
\end{cor}
\subsection{For $\kappa=(Q_4,0,3)$}
It follows from Theorem \ref{cisoforderm-2} that $3\mathbb{Z}$ is a CIS of order 2 for $V(Q_4)$.  In this case,
$$\Psi_{\kappa}(t)=
e^{2\pi\mathrm{i}t}\begin{bmatrix}
\dfrac{1}{6} & \dfrac{2}{3} & \dfrac{1}{6}\\
-\dfrac{1}{2} & 0 & \dfrac{1}{2}\\
1 & -2 & 1
\end{bmatrix}
~\text{and}~
\Psi_{\kappa}^{-1}(t)=
e^{-2\pi\mathrm{i}t}\begin{bmatrix}
1 & -1 & \dfrac{1}{3}\\
1 & 0 & -\dfrac{1}{6}\\
1 & 1 & \dfrac{1}{3}
\end{bmatrix}.
$$
We now obtain from \eqref{thetaroof} that
\begin{eqnarray*}
\hspace{-1.5cm}\widehat{\Theta_{0,\kappa}}(\xi)
=\widehat{Q_4}(\xi)(z^3+z^2+z),~ \widehat{\Theta_{1,\kappa}}(\xi)=\widehat{Q_4}(\xi)(-z^3+z),
\end{eqnarray*}
and
\begin{eqnarray*}
\hspace{-5cm}\widehat{\Theta_{2,\kappa}}(\xi)=\widehat{Q_4}(\xi)\left(\frac{1}{3}z^3-\frac{1}{6}z^2+\frac{1}{3}z\right).
\end{eqnarray*}
Hence
\begin{equation*}
\hspace{-4.75cm}\widehat{\Theta_{0,\kappa}}^{\prime}(\xi)
=2\pi\mathrm{i}\widehat{Q_4}(\xi)(3z^3+2z^2+z)+\widehat{Q_4}^{\prime}(\xi)(z^3+z^2+z),
\end{equation*}
\begin{equation*}
\hspace{-6.05cm}\widehat{\Theta_{1,\kappa}}^{\prime}(\xi)
=2\pi\mathrm{i}\widehat{Q_4}(\xi)(-3z^3+z)+\widehat{Q_4}^{\prime}(\xi)(-z^3+z),
\end{equation*}
\begin{equation*}
\widehat{\Theta_{0,\kappa}}^{\prime\prime}(\xi)
=-4\pi^2\widehat{Q_4}(\xi)(9z^3+4z^2+z)+4\pi\mathrm{i}\widehat{Q_4}^{\prime}(\xi)(3z^3+2z^2+z)+\widehat{Q_4}^{\prime\prime}(\xi)(z^3+z^2+z),
\end{equation*}
\begin{multline*}
\hspace{-0.34cm}\widehat{\Theta_{0,\kappa}}^{\prime\prime\prime}(\xi)=-8\pi^3\mathrm{i}\widehat{Q_4}(\xi)(27z^3+8z^2+z)-12\pi^2\widehat{Q_4}^{\prime}(\xi)(9z^3+4z^2+z)\\
+6\pi\mathrm{i}\widehat{Q_4}^{\prime\prime}(\xi)(3z^3+2z^2+z)+\widehat{Q_4}^{\prime\prime\prime}(\xi)(z^3+z^2+z),
\end{multline*}
\begin{equation*}
\hspace{-1.3cm}\widehat{\Theta_{1,\kappa}}^{\prime\prime}(\xi)=-4\pi^2\widehat{Q_4}(\xi)(-9z^3+z)+4\pi\mathrm{i}\widehat{Q_4}^{\prime}(\xi)(-3z^3+z)+\widehat{Q_4}^{\prime\prime}(\xi)(-z^3+z),
\end{equation*}
and
\begin{equation*}
\hspace{-2.8cm}\widehat{\Theta_{2,\kappa}}^{\prime}(\xi)=2\pi\mathrm{i}\widehat{Q_4}(\xi)\left(z^3-\frac{1}{3}z^2+\frac{1}{3}z\right)+\widehat{Q_4}^{\prime}(\xi)\left(\frac{1}{3}z^3-\frac{1}{6}z^2+\frac{1}{3}z\right).
\end{equation*}
It is easy to compute that $\widehat{Q_4}(0)=1,$ $\widehat{Q_4}^{\prime}(0)=-4\pi\mathrm{i},$ $\widehat{Q_4}^{\prime\prime}(0)=-\frac{52}{3}\pi^2,$ and $\widehat{Q_4}^{\prime\prime\prime}(0)=80\pi^3\mathrm{i}.$ Again, by Strang-Fix condition of $Q_4,$ we verified that
\begin{equation*}
\hspace{-8cm}\widehat{\Theta_{0,\kappa}}(l/3)=
\begin{cases}
3 &~\text{if}~ l=0,\\
0 &~\text{if}~ l\in\mathbb{Z}\setminus\{0\},
\end{cases}
\end{equation*}
\begin{equation*}
\hspace{-2cm}\widehat{\Theta_{1,\kappa}}(l/3)=\frac{\mathrm{i}}{2\pi}\widehat{\Theta_{0,\kappa}}^{\prime}(l/3)=
\begin{cases}
0 &~\text{if}~l\in3\mathbb{Z},\\
\tfrac{1}{2}(-3+\sqrt{3}\mathrm{i})\widehat{Q_4}(\tfrac{l}{3}) &~\text{if}~l\in1+3\mathbb{Z},\\
\tfrac{1}{2}(-3-\sqrt{3}\mathrm{i})\widehat{Q_4}(\tfrac{l}{3}) &~\text{if}~l\in2+3\mathbb{Z},
\end{cases}
\end{equation*}
\begin{equation*}
\hspace{0.7cm}\widehat{\Theta_{0,\kappa}}^{\prime\prime}(l/3)=
\begin{cases}
-12\pi^2 &~\text{if}~l=0,\\
0&~\text{if}~l\in3\mathbb{Z}\setminus\{0\},\\
-2\pi^2(13-3\sqrt{3}\mathrm{i})\widehat{Q_4}\left(\tfrac{l}{3}\right)+2\pi\mathrm{i}(3-\sqrt{3}\mathrm{i})\widehat{Q_4}^{\prime}\left(\tfrac{l}{3}\right) &~\text{if}~l\in1+3\mathbb{Z},\\
-2\pi^2(13+3\sqrt{3}\mathrm{i})\widehat{Q_4}\left(\tfrac{l}{3}\right)+2\pi\mathrm{i}(3+\sqrt{3}\mathrm{i})\widehat{Q_4}^{\prime}\left(\tfrac{l}{3}\right) &~\text{if}~l\in2+3\mathbb{Z},
\end{cases}
\end{equation*}
\begin{equation*}
\hspace{-0.5cm}\widehat{\Theta_{1,\kappa}}^{\prime}(l/3)=
\begin{cases}
-4\pi\mathrm{i} &~\text{if}~l=0,\\
0 &~\text{if}~l\in3\mathbb{Z}\setminus\{0\},\\
\pi\mathrm{i}(-7+\sqrt{3}\mathrm{i})\widehat{Q_4}\left(\tfrac{l}{3}\right)+\tfrac{-3+\sqrt{3}\mathrm{i}}{2}\widehat{Q_4}^{\prime}\left(\tfrac{l}{3}\right)&~\text{if}~l\in1+3\mathbb{Z},\\
\pi\mathrm{i}(-7-\sqrt{3}\mathrm{i})\widehat{Q_4}\left(\tfrac{l}{3}\right)+\tfrac{3+\sqrt{3}\mathrm{i}}{2}\widehat{Q_4}^{\prime}\left(\tfrac{l}{3}\right) &~\text{if}~l\in2+3\mathbb{Z},
\end{cases}
\end{equation*}
\begin{equation*}
\hspace{-5cm}\widehat{\Theta_{2,\kappa}}(l/3)=
\begin{cases}
\tfrac{1}{2} &~\text{if}~l=0,\\
0&~\text{if}~l\in3\mathbb{Z}\setminus\{0\},\\
\tfrac{1+\sqrt{3}\mathrm{i}}{4}\widehat{Q_4}\left(\tfrac{l}{3}\right) &~\text{if}~l\in1+3\mathbb{Z},\\
\tfrac{1-\sqrt{3}\mathrm{i}}{4}\widehat{Q_4}\left(\tfrac{l}{3}\right) &~\text{if}~l\in2+3\mathbb{Z},
\end{cases}
\end{equation*}
\begin{equation*}
\widehat{\Theta_{2,\kappa}}^{\prime}(l/3)=
\begin{cases}
0&~\text{if}~l\in3\mathbb{Z},\\
2\pi\mathrm{i}\left(1+\tfrac{\sqrt{3}\mathrm{i}}{3}\right)\widehat{Q_4}\left(\tfrac{l}{3}\right)+\frac{1}{4}(1+\sqrt{3}\mathrm{i})\widehat{Q_4}^{\prime}\left(\tfrac{l}{3}\right)&~\text{if}~l\in1+3\mathbb{Z},\\
2\pi\mathrm{i}\left(1-\tfrac{\sqrt{3}\mathrm{i}}{3}\right)\widehat{Q_4}\left(\tfrac{l}{3}\right)+\frac{1}{4}(1-\sqrt{3}\mathrm{i})\widehat{Q_4}^{\prime}\left(\tfrac{l}{3}\right) &~\text{if}~l\in2+3\mathbb{Z},
\end{cases}
\end{equation*}
\begin{equation*}
\hspace{.7cm}\widehat{\Theta_{0,\kappa}}^{\prime\prime\prime}(l/3)=
\begin{cases}
0&~\text{if}~l\in3\mathbb{Z},\\
-4\pi^3\mathrm{i}\gamma\widehat{Q_4}\left(\tfrac{l}{3}\right)-6\pi^2\beta\widehat{Q_4}^{\prime}\left(\tfrac{l}{3}\right)+3\pi\mathrm{i}\alpha\widehat{Q_4}^{\prime\prime}\left(\tfrac{l}{3}\right)&~\text{if}~l\in1+3\mathbb{Z},\\
-4\pi^3i\overline{\gamma}\widehat{Q_4}\left(\tfrac{l}{3}\right)-6\pi^2\overline{\beta}\widehat{Q_4}^{\prime}\left(\tfrac{l}{3}\right)+3\pi\mathrm{i}\overline{\alpha}\widehat{Q_4}^{\prime\prime}\left(\tfrac{l}{3}\right)&~\text{if}~l\in2+3\mathbb{Z},
\end{cases}
\end{equation*}
and
\begin{equation*}
\hspace{-0.3cm}\widehat{\Theta_{1,\kappa}}^{\prime\prime}(l/3)=
\begin{cases}
0&~\text{if}~l\in3\mathbb{Z},\\
2\pi^2a\widehat{Q_4}\left(\tfrac{l}{3}\right)-2\pi\mathrm{i}b\widehat{Q_4}^{\prime}\left(\tfrac{l}{3}\right)-\frac{1}{2}\alpha\widehat{Q_4}^{\prime\prime}\left(\tfrac{l}{3}\right)&~\text{if}~l\in1+3\mathbb{Z},\\
2\pi^2\overline{a}\widehat{Q_4}\left(\tfrac{l}{3}\right)-2\pi\mathrm{i}\overline{b}\widehat{Q_4}^{\prime}\left(\tfrac{l}{3}\right)-\frac{1}{2}\overline{\alpha}\widehat{Q_4}^{\prime\prime}\left(\tfrac{l}{3}\right)&~\text{if}~l\in2+3\mathbb{Z},
\end{cases}
\end{equation*}
where $\alpha=3-\sqrt{3}\mathrm{i},$ $\beta=13-3\sqrt{3}\mathrm{i},$ $\gamma=45-7\sqrt{3}\mathrm{i},$ $a=19-\sqrt{3}\mathrm{i},$ and $b=7-\sqrt{3}\mathrm{i}.$

By substituting the values of $\widehat{\Theta_{i,\kappa}}^{(r)}(l/3)$ in \eqref{ftthetacond}, we can show that the operator $S_W^{\kappa}$ satisfies the reproducing polynomial property of order $3.$ Again from Corollary \ref{approxcor}, we obtain the following result.
\begin{cor}
Let $\kappa=(Q_4,0,3).$ If $f\in\Lambda_{3}^p$ such that $\tau_3(f^{(i)},\delta)_p=\mathcal{O}(\delta^{\alpha})$ as $\delta\to0+$ for each $i=0,1,2,$ then
\begin{equation}
\|S_W^{\kappa}f-f\|_{L^p(\mathbb{R})}= \mathcal{O}(W^{-\alpha})~ (W\to\infty).
\end{equation}
\end{cor}
\subsection{For $\kappa=(Q_4,\tfrac{1}{2},2)$}
Using the properties of $B$- splines, we get
$$\Psi_{\kappa}(t)
=\begin{bmatrix}
\frac{1}{48}(1+23e^{2\pi\mathrm{i}t}) & \frac{e^{2\pi\mathrm{i}t}}{48}(23+e^{2\pi\mathrm{i}t})\\ \\
\frac{1}{8}(1-5e^{2\pi\mathrm{i}t}) & \frac{e^{2\pi\mathrm{i}t}}{8}(5-e^{2\pi\mathrm{i}t})
\end{bmatrix}
=\begin{bmatrix}
\frac{1}{48}(1+23w) & \frac{w}{48}(23+w)\\ \\
\frac{1}{8}(1-5w) & \frac{w}{8}(5-w)
\end{bmatrix}$$     
and
$\det\Psi_{\kappa}(t)=\frac{w}{64}(-3+38w-3w^2)$, where $w=e^{2\pi\mathrm{i}t}$. Since $|\det\Psi_{\kappa}(t)|\neq0$, $\tfrac{1}{2}+2\mathbb{Z}$ is a CIS of order $1$ for $V(Q_4)$. The inverse of $\Psi_{\kappa}$  is  given by
$$\Psi^{-1}_{\kappa}(t)= \tfrac{64}{3w^2(\frac{38}{3}-w-w^{-1})}
\begin{bmatrix}
\frac{w(5-w)}{8} & \frac{-w(23+w)}{48}\\ \\
\frac{-(1-5w)}{8}& \frac{1+23w}{48}
\end{bmatrix}$$
and its Fourier coefficients are given by
$$\widehat{\Psi^{-1}_{\kappa}}(n)=
\begin{bmatrix}
\frac{8r}{3(1-r^2)}(5r^{|n+1|}-r^{|n|}) & \frac{-4r}{9(1-r^2)}(23r^{|n+1|}+r^{|n|}) \\ \\
\frac{-8r}{3(1-r^2)}(r^{|n+2|}-5r^{|n+1|}) & \frac{4r}{9(1-r^2)}(r^{|n+2|}+23r^{|n+1|})
\end{bmatrix},~r=\tfrac{19-4\sqrt{22}}{3}. $$ 
Consequently, we have
\begin{eqnarray*}
\Theta_{0,\kappa}(t)
&=&\sum\limits_{v\in\mathbb{Z}}\Big[\widehat{(\Psi^{-1}_{\kappa})^{00}}(v) Q_4(t-2v)+
\widehat{(\Psi^{-1}_{\kappa})^{10}}(v) Q_4(t-2v-1)\Big]\\
&=&\tfrac{8r}{3(1-r^2)}\sum\limits_{v\in\mathbb{Z}}\Big[(5r^{|v+1|}-r^{|v|})Q_4(t-2v)+(5r^{|v+1|}-r^{|v+2|})Q_4(t-2v-1)\Big]\\
\end{eqnarray*} and
\begin{eqnarray*}
\Theta_{1,\kappa}(t)
&=&\sum\limits_{v\in\mathbb{Z}}\Big[\widehat{(\Psi^{-1}_{\kappa})^{01}}(v) Q_4(t-2v)+
\widehat{(\Psi^{-1}_{\kappa})^{11}}(v) Q_4(t-2v-1)\Big]\\
&=&\tfrac{-4r}{9(1-r^2)}\\
&&\times
\sum_{v\in\mathbb{Z}}\Big[(23r^{|v+1|}+r^{|v|})Q_4(t-2v)-(23r^{|v+1|}+r^{|v+2|})Q_4(t-2v-1)\Big].\\
\end{eqnarray*}
Hence every $f\in V(Q_4)$ can be written as
\begin{eqnarray*}
f(t)
&=&\sum_{l\in\mathbb{Z}}\Big[f\Big(\tfrac{1}{2}+2l\Big)\Theta_{0,\kappa}(t-2l)+
f^{\prime}\Big(\tfrac{1}{2}+
2l\Big)\Theta_{1,\kappa}(t-2l)\Big].\nonumber\\
\end{eqnarray*}
We now obtain from \eqref{thetaroof} that
$$\widehat{\Theta_{0,\kappa}}(\xi)=g_0(\xi)\widehat{Q_4}(\xi)~ \text{and}~ \widehat{\Theta_{1,\kappa}}(\xi)=g_1(\xi)\widehat{Q_4}(\xi),$$
and hence
\begin{eqnarray*}
\hspace{-5.4cm}\widehat{\Theta_{0,{\kappa}}}^{\prime}(\xi)
=g_0^{\prime}(\xi)\widehat{Q_4}(\xi)+g_0(\xi)\widehat{Q_4}^{\prime}(\xi),
\end{eqnarray*}
\begin{equation*}
\hspace{-3.1cm}\widehat{\Theta_{0,\kappa}}^{\prime\prime}(\xi)
=g_0^{\prime\prime}(\xi)\widehat{Q_4}(\xi)+g_0(\xi)\widehat{Q_4}^{\prime\prime}(\xi)+2g_0^{\prime}(\xi)\widehat{Q_4}^{\prime}(\xi),
\end{equation*}
\begin{equation*}
\widehat{\Theta_{0,\kappa}}^{\prime\prime\prime}(\xi)=g_0^{\prime\prime\prime}(\xi)\widehat{Q_4}(\xi)+3g_0^{\prime\prime}(\xi)\widehat{Q_4}^{\prime}(\xi)+3g_0^{\prime}(\xi)\widehat{Q_4}^{\prime\prime}(\xi)+g_0(\xi)\widehat{Q_4}^{\prime\prime\prime}(\xi),
\end{equation*}
\begin{equation*}
\hspace{-5.9cm}\widehat{\Theta_{1,{\kappa}}}^{\prime}(\xi)
=g_1^{\prime}(\xi)\widehat{Q_4}(\xi)+g_1(\xi)\widehat{Q_4}^{\prime}(\xi),
\end{equation*}
\begin{equation*}
\hspace{-3cm}\widehat{\Theta_{1,{\kappa}}}^{\prime\prime}(\xi)
=g_1^{\prime\prime}(\xi)\widehat{Q_4}(\xi)+2g_1^{\prime}(\xi)\widehat{Q_4}^{\prime}(\xi)+g_1(\xi)\widehat{Q_4}^{\prime\prime}(\xi),
\end{equation*}
where
$$g_0(\xi)=\dfrac{8z^2(-1+5z+5z^{2}-z^3)}{38z^{2}-3z^{4}-3},~
g_1(\xi)=\dfrac{4z^2(-1+23z-23z^{2}+z^{3})}{3(38z^{2}-3z^{4}-3)}.$$ 
\begin{equation*}
\widehat{\Theta_{0,{\kappa}}}(l/2)=
\begin{cases}
2 &~\text{if}~l=0,\\
0 &~\text{if}~ l\in\mathbb{Z}\setminus\{0\},
\end{cases}\nonumber~
\widehat{\Theta_{0,{\kappa}}}^{\prime}(l/2)=
\begin{cases}
-2\pi\mathrm{i} &~\text{if}~l=0,\\
0 &~\text{if}~l\in2\mathbb{Z}\setminus\{0\},\\
4\pi\mathrm{i}\widehat{Q_4}\left(\tfrac{l}{2}\right) &~\text{if}~l\in1+2\mathbb{Z},
\end{cases}
\end{equation*}
Therefore
\begin{equation*}
\hspace{-3.3cm}\widehat{\Theta_{0,{\kappa}}}^{\prime\prime}(l/2)=
\begin{cases}
-\tfrac{26}{3}\pi^2 &~\text{if} ~l=0,\\
0 &~\text{if}~ l\in2\mathbb{Z}\setminus\{0\},\\
8\pi\mathrm{i}\widehat{Q_4}^{\prime}\left(\tfrac{l}{2}\right)-24\pi^2\widehat{Q_4}\left(\tfrac{l}{2}\right) &~\text{if}~ l\in1+2\mathbb{Z},
\end{cases}
\end{equation*}
\begin{equation*}
\widehat{\Theta_{0,{\kappa}}}^{\prime\prime\prime}(l/2)=
\begin{cases}
22\pi^3\mathrm{i} &~\text{if} ~l=0,\\
0 &~\text{if}~ l\in2\mathbb{Z}\setminus\{0\},\\
12\pi\mathrm{i}\widehat{Q_4}^{\prime\prime}\left(\tfrac{l}{2}\right)-72\pi^2\widehat{Q_4}^{\prime}\left(\tfrac{l}{2}\right)-160\pi^3\mathrm{i}\widehat{Q_4}\left(\tfrac{l}{2}\right) &~\text{if}~ l\in1+2\mathbb{Z},
\end{cases}
\end{equation*}
\begin{equation*}
\hspace{-6.5cm}\widehat{\Theta_{1,{\kappa}}}(l/2)=
\begin{cases}
0 &~\text{if}~ l\in2\mathbb{Z},\\
-2\widehat{Q_4}\left(\tfrac{l}{2}\right)~&~\text{if}~ l\in1+2\mathbb{Z},
\end{cases}
\end{equation*}
\begin{equation*}
\hspace{-3.7cm}\widehat{\Theta_{1,{\kappa}}}^{\prime}(l/2)=
\begin{cases}
-\tfrac{5}{3}\pi\mathrm{i} &~\text{if} ~l=0,\\
0 &~\text{if}~ l\in2\mathbb{Z}\setminus\{0\},\\
-2\widehat{Q_4}^{\prime}\left(\tfrac{l}{2}\right)-6\pi\mathrm{i}\widehat{Q_4}\left(\tfrac{l}{2}\right) &~\text{if}~  l\in1+2\mathbb{Z},
\end{cases}
\end{equation*}
and
\begin{equation*}
\hspace{-0.7cm}\widehat{\Theta_{1,{\kappa}}}^{\prime\prime}(l/2)=
\begin{cases}
-\tfrac{10}{3}\pi^2&~\text{if} ~l=0,\\
0 &~\text{if}~ l\in2\mathbb{Z}\setminus\{0\},\\
-2\widehat{Q_4}^{\prime\prime}\left(\tfrac{l}{2}\right)-12\pi\mathrm{i}\widehat{Q_4}^{\prime}\left(\tfrac{l}{2}\right)+\frac{80\pi^2}{3}\widehat{Q_4}\left(\tfrac{l}{2}\right) &~\text{if}~  l\in1+2\mathbb{Z}.
\end{cases}
\end{equation*}
By substituting the values of $\widehat{\Theta_{i,\kappa}}^{(r)}(l/2)$ in \eqref{ftthetacond}, we can show that the operator $S_W^{\kappa}$ satisfies the reproducing polynomial property of order 3. Hence we obtain the following result from Corollary \ref{approxcor}.
\begin{cor}
Let $\kappa=(Q_4,\frac{1}{2},2).$ If $f\in\Lambda_{2}^p$ such that  $\tau_3(f^{(i)},\delta)_p=\mathcal{O}(\delta^{\alpha})$ as $\delta\to0+$ for each $i=0,1,$ then
\begin{equation}
\|S_W^{\kappa}f-f\|_{L^p(\mathbb{R})}= \mathcal{O}(W^{-\alpha})~ (W\to\infty).
\end{equation}
\end{cor}

We now approximate three real-valued functions $f_k, k=1,2,3$ using derivative sampling expansions.
The function $f_1 :\mathbb{R}\to \mathbb{R}$ is given by
$$f_1(t)=e^{-t^2/4}\sin(2\pi t).$$
It is an infinitely differentiable function and $f_1\in\Lambda_{\rho}^p$ for any $\rho\in\mathbb{N}.$
The function $f_2 :\mathbb{R}\to\mathbb{R}$ defined by 
$$f_2(t)=\begin{cases}
\sin^2\pi t&~\text{if}~ |t|\leq3,\\
0&~ \text{elsewhere},
\end{cases}$$
is piecewise continuously differentiable but its second derivative is no longer continuous. Further, $f_2\in\Lambda_3^p.$
Finally, we consider the discontinuous function
$$f_3(t)=\begin{cases}
-\frac{1}{2}t^3+2&~\text{if}~ -1.5<t<3,\\
0&~ \text{elsewhere}.
\end{cases}$$
It is an infinitely  differentiable function except at $t=-1.5$ and $t=3.$ Therefore, we can not say that $f_3\in\Lambda_{\rho}^p$ for every $\rho\geq2$ and $W\geq1$. However, for any irrational number $q$ if we choose the subcollection $\Sigma_{Nq}:= \Big(\tfrac{a+\rho l}{Nq}\Big)_{l\in\mathbb{Z}},$ $N\in\mathbb{N}$ from $\Sigma_{W}$  
and replace the space  $\Lambda_\rho^p$ by $$\tilde{\Lambda}_{\rho}^p:=\left\{f\in M(\mathbb{R}): \sum\limits_{l\in\mathbb{Z}}\sum\limits_{i=0}^{\rho-1}\left|f^{(i)}\left(\frac{a+\rho l}{Nq}\right)\right|^p<\infty~ \text{for any} ~N\in\mathbb{N}\right\},$$ we can show that $f_3\in\tilde{\Lambda}_{\rho}^p$ for any $\rho\geq1.$ We can realize that for any discontinuous function, we need to choose a subcollection from $\Sigma_W$ and its corresponding space in a suitable way.
Arguing as in \cite{BBSV2}, we can show that as $\delta\to0+$
\begin{itemize}
\item [$(i)$] $\tau_r(f_1^{(i)};\delta)_p=\mathcal{O}(\delta^r),$
for any $r\in\mathbb{N},~ i=0,1,\dots;$
\item[$(ii)$] 
$\tau_r(f_2;\delta)_p=
\begin{cases}
\mathcal{O}(\delta^r)~&\text{if}~r=1,2,\\
\mathcal{O}(\delta^{2+1/p})~&\text{if}~r\geq3,
\end{cases}$
$\tau_r(f_2^{\prime};\delta)_p=
\begin{cases}
\mathcal{O}(\delta)~&\text{if}~r=1,\\
\mathcal{O}(\delta^{1+1/p})~&\text{if}~r\geq2,
\end{cases}$ and $\tau_r(f_2^{\prime\prime};\delta)_p=\mathcal{O}(\delta^{1/p}),$ for any $r\in\mathbb{N};$
\item[$(iii)$] $\tau_r(f_3;\delta)_p=\mathcal{O}(\delta^{1/p}),$ for any $r\in\mathbb{N}.$ 
\end{itemize}
We now find the $\tau-$modulus for higher derivatives of $f_2$ and $f_3.$ Notice that $f_2$ and $f_3$ are infinitely differentiable functions except at two points. Now
$$f_3^{\prime}(t)=\begin{cases}
-\frac{3}{2}t^2~&\text{if}~-1.5<t<3,\\
0~&~\text{elsewhere except at}~ t=-1.5~ \text{and}~ t=3.
\end{cases}$$
In order to evaluate $\tau_1(f_3^{\prime},\delta),$ let
$$A_1:=\left[-\tfrac{3}{2}-\delta, -\tfrac{3}{2}+\delta\right]\cup[3-\delta, 3+\delta]\setminus\{-1.5,3\},$$
$$\hspace{-5cm}A_2:=\left[-\tfrac{3}{2}+\delta, 3-\delta\right],$$ 
and
$$\hspace{-5cm}I_{\delta/2}(t):=\left[t-\tfrac{\delta}{2}, t+\tfrac{\delta}{2}\right].$$
If $t\notin A_1\cup A_2,$ then it is clear that
$$\sup\limits_{x,x+h\in I_{\delta/2}(t)}|f_3^{\prime}(x+h)-f_3^{\prime}(x)|=0.$$
For $t\in A_1,$ we have
$$\sup\limits_{x,x+h\in I_{\delta/2}(t)}|f_3^{\prime}(x+h)-f_3^{\prime}(x)|\leq2\sup\limits_{x\in\mathbb{R}}|f_3^{\prime}(x)|\leq\frac{27}{2}.$$ 
Finally, for $t\in A_2,$ 
by the mean value theorem, we have
$$\sup\limits_{x,x+h\in I_{\delta/2}(t)}|f_3^{\prime}(x+h)-f_3^{\prime}(x)|\leq9\delta$$ noting that $x,x+h\in I_{\delta/2}(t)$ implies $|h|\leq\delta.$ Hence we obtain for the local modulus
$$\omega_1(f_3^{\prime};t;\delta)\begin{cases}
\leq\frac{27}{2}~&\text{if}~t\in A_1,\\
\leq9\delta~&\text{if}~t\in A_2,\\
=0 ~&~\text{elsewhere}.
\end{cases}$$
Consequently, for $1\leq p<\infty,$ $0<\delta<1,$ and some constant $C>0$, we get
\begin{eqnarray*}
\tau_1(f_3^{\prime};\delta)_p^p
&\leq&\int\limits_{A_1}\left(\frac{27}{2}\right)^p~dt+\int\limits_{A_2}(9\delta)^p~dt\\
&=&2\left(\frac{27}{2}\right)^p\delta+(9\delta)^p\left(\frac{9}{2}-\delta\right)\leq C\delta,
\end{eqnarray*}
which implies $\tau_1(f_3^{\prime};\delta)_p=\mathcal{O}(\delta^{1/p}),$ $\delta\to0+.$
It is known that
\begin{equation*}
\tau_r(f;\delta)_p\leq2^s\tau_{r-s}\left(f;\frac{r}{r-s}\delta\right)_p~(s=1,2,\dots,r-1)
\end{equation*}
and
\begin{equation}\label{proptau}
\hspace{-1.7cm}\tau_r(f;\lambda\delta)_p\leq(2(\lambda+1))^{r+1}\tau_{r}(f;\delta)_p~(\lambda>0),
\end{equation}
for each $r\in\mathbb{N}$ (see \cite{BBSV1}). Hence
$$\tau_r(f_3^{\prime};\delta)_p=\mathcal{O}(\delta^{1/p})~(\delta\to0+).$$
Similarly, we can show that  $$\tau_r(f_3^{(i)};\delta)_p=\mathcal{O}(\delta^{1/p})~(\delta\to0+),~\text{for any}~r\in\mathbb{N}~\text{and}~i=1,2,\dots.$$
Using the same argument, we can show that $\tau_r(f_2^{(i)};\delta)_p=\mathcal{O}(\delta^{1/p})~(\delta\to0+),$ for any $r\in\mathbb{N}$ and $i=3,4,\dots.$
Figures \ref{q3f1}-\ref{q4f3} show the approximation of the functions $f_k$, $k=1,2,3,$ by the series $S_W^{\kappa}f$ based on $\kappa=(Q_3,0,2)$ and $\kappa=(Q_4,0,3).$

\begin{figure}[H]
     \centering
     \begin{subfigure}[b]{0.3\textwidth}
         \centering
         \includegraphics[width=\textwidth]{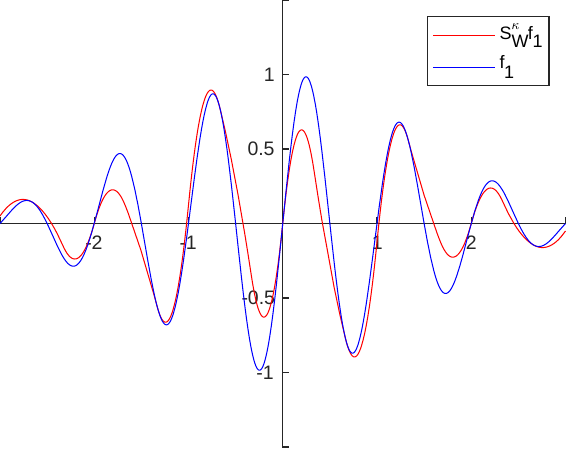}
         \caption{Graphs of $f_1$  and $S_W^{\kappa}f_1$ with $W=3.$}
         \label{q3w3f1}
     \end{subfigure}
     \hfill
     \begin{subfigure}[b]{0.3\textwidth}
         \centering
         \includegraphics[width=\textwidth]{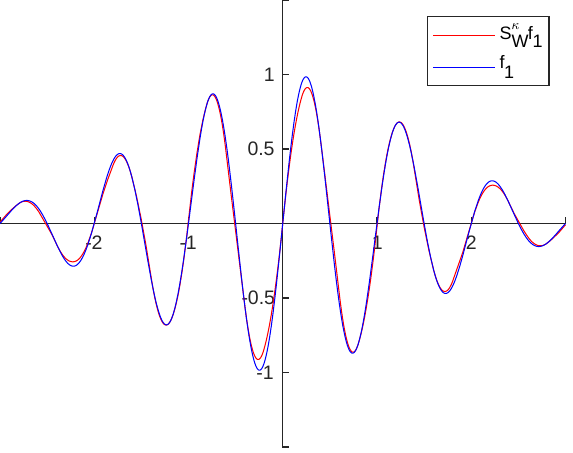}
         \caption{Graphs of $f_1$ and $S_W^{\kappa}f_1$ with $W=5.$}
         \label{q3w5f1}
     \end{subfigure}
     \newline
     \begin{subfigure}[b]{0.3\textwidth}
         \centering
         \includegraphics[width=\textwidth]{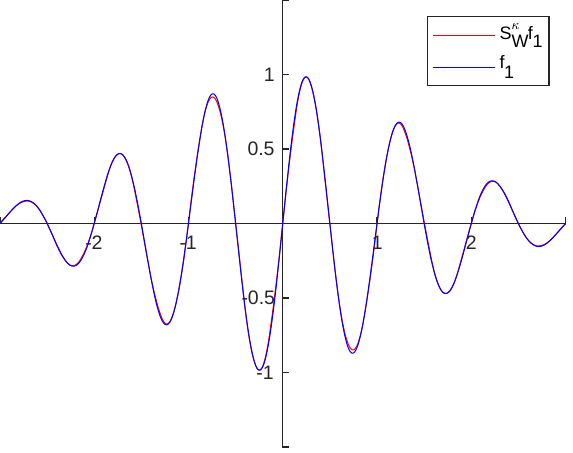}
         \caption{Graphs of $f_1$ and $S_W^{\kappa}f_1$ with $W=7.$}
         \label{q3w7f1}
     \end{subfigure}
     \hfill
     \begin{subfigure}[b]{0.3\textwidth}
         \centering
         \includegraphics[width=\textwidth]{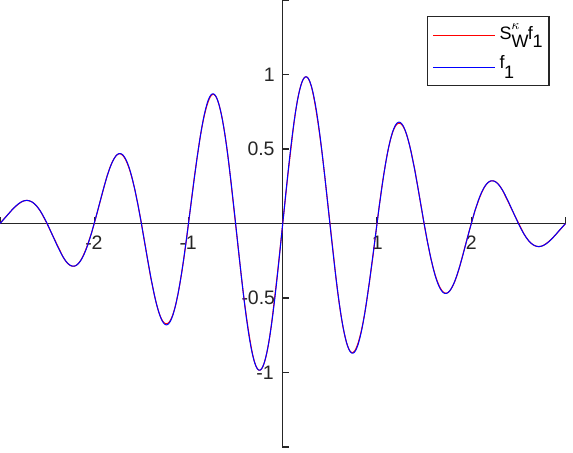}
         \caption{Graphs of $f_1$ and $S_W^{\kappa}f_1$ with $W=9.$}
         \label{q3w9f1}
     \end{subfigure}
        \caption{Approximation of $f_1$ based  on $S_W^{\kappa}f_1$  for $\kappa=(Q_3,0,2).$}
        \label{q3f1}
\end{figure}
\begin{figure}[H]
     \centering
     \begin{subfigure}[b]{0.3\textwidth}
         \centering
         \includegraphics[width=\textwidth]{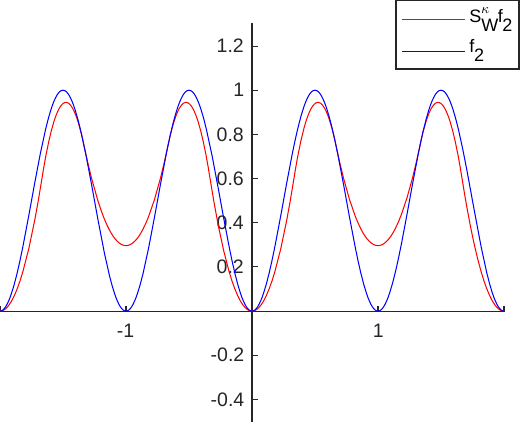}
         \caption{Graphs of $f_2$ and $S_W^{\kappa}f_2$ with $W=3.$}
         \label{q3w3f2}
     \end{subfigure}
     \hfill
     \begin{subfigure}[b]{0.3\textwidth}
         \centering
         \includegraphics[width=\textwidth]{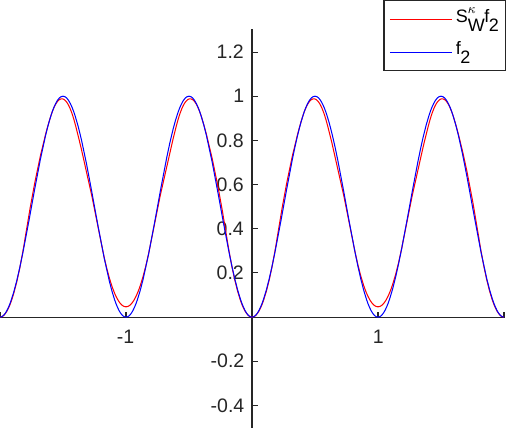}
         \caption{Graphs of $f_2$ and $S_W^{\kappa}f_2$ with $W=5.$}
         \label{q3w5f2}
     \end{subfigure}
     \newline
     \begin{subfigure}[b]{0.3\textwidth}
         \centering
         \includegraphics[width=\textwidth]{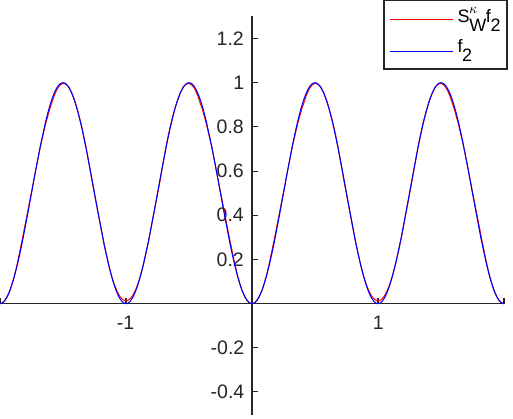}
         \caption{Graphs of $f_2$ and $S_W^{\kappa}f_2$ with $W=7.$}
         \label{q3w7f2}
     \end{subfigure}
     \hfill
     \begin{subfigure}[b]{0.3\textwidth}
         \centering
         \includegraphics[width=\textwidth]{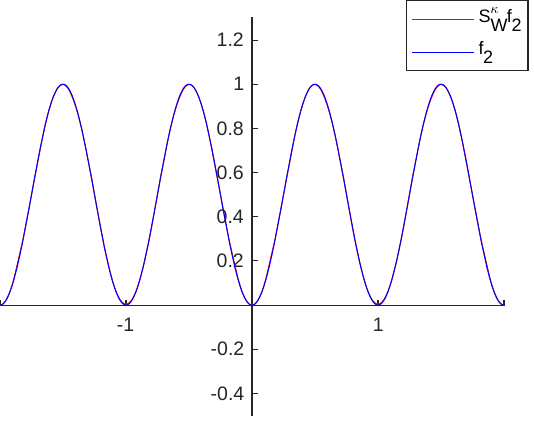}
         \caption{Graphs of $f_2$ and $S_W^{\kappa}f_2$ with $W=9.$}
         \label{q3w9f2}
     \end{subfigure}
        \caption{Approximation of $f_2$ based  on $S_W^{\kappa}f_2$  for $\kappa=(Q_3,0,2).$}
        \label{q3f2}
\end{figure}

\begin{figure}[H]
     \centering
     \begin{subfigure}[b]{0.3\textwidth}
         \centering
         \includegraphics[width=\textwidth]{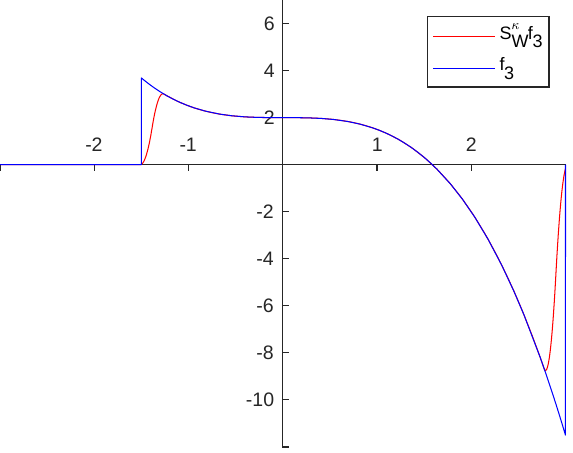}
         \caption{Graphs of $f_3$ and $S_W^{\kappa}f_3$ with $W=3\sqrt{7}.$}
         \label{q3w3sqrt7f3}
     \end{subfigure}
     \hfill
     \begin{subfigure}[b]{0.3\textwidth}
         \centering
         \includegraphics[width=\textwidth]{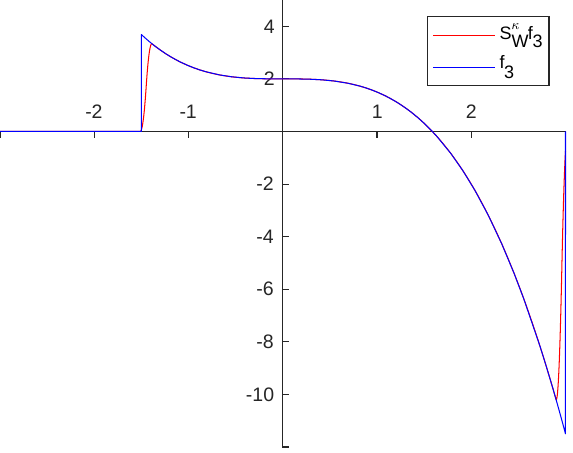}
         \caption{Graphs of $f_3$ and $S_W^{\kappa}f_3$ with $W=6\sqrt{7}.$}
         \label{q3w6sqrt7f3}
     \end{subfigure}
     \newline
     \begin{subfigure}[b]{0.3\textwidth}
         \centering
         \includegraphics[width=\textwidth]{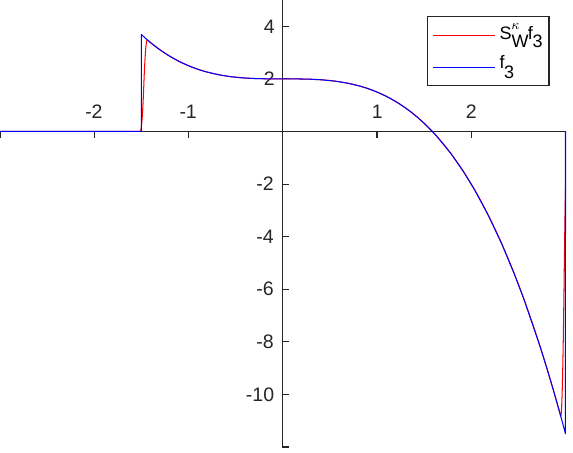}
         \caption{Graphs of $f_3$ and $S_W^{\kappa}f_3$ with $W=10\sqrt{7}.$}
         \label{q3w10sqrt7f3}
     \end{subfigure}
     \hfill
     \begin{subfigure}[b]{0.3\textwidth}
         \centering
         \includegraphics[width=\textwidth]{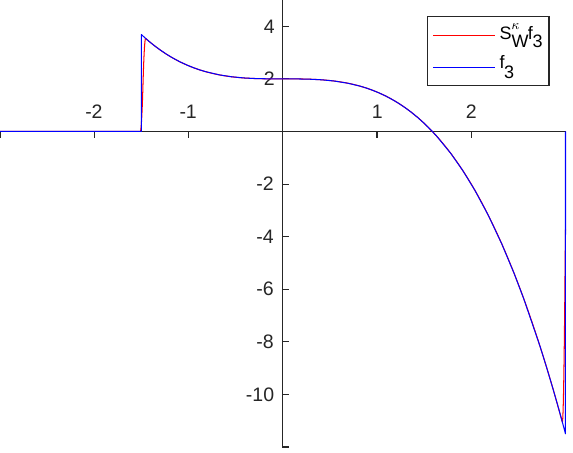}    
         \caption{Graphs of $f_3$  and $S_W^{\kappa}f_3$ with $W=13\sqrt{7}.$}
         \label{q3w13sqrt7f3}
     \end{subfigure}
        \caption{Approximation of $f_3$ based  on $S_W^{\kappa}f_3$  for $\kappa=(Q_3,0,2).$}
        \label{q3f3}
\end{figure}
\begin{figure}[H]
     \centering
     \begin{subfigure}[b]{0.3\textwidth}
         \centering
         \includegraphics[width=\textwidth]{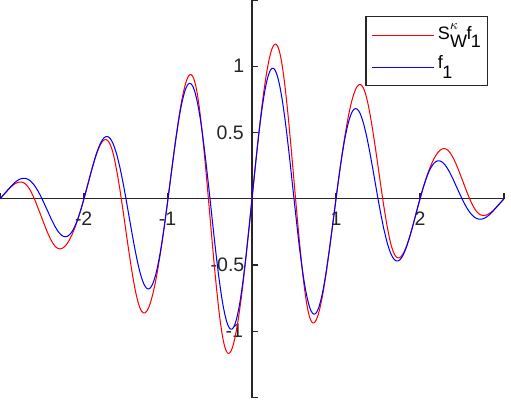}
         \caption{Graphs of $f_1$ and $S_W^{\kappa}f_1$ with $W=3.$}
         \label{q4w3f1}
     \end{subfigure}
     \hfill
     \begin{subfigure}[b]{0.3\textwidth}
         \centering
         \includegraphics[width=\textwidth]{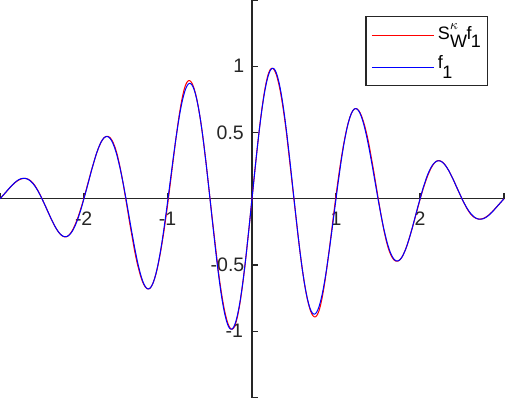}
         \caption{Graphs of $f_1$ and $S_W^{\kappa}f_1$ with $W=5.$}
         \label{q4w5f1}
     \end{subfigure}
     \newline
     \begin{subfigure}[b]{0.3\textwidth}
         \centering
         \includegraphics[width=\textwidth]{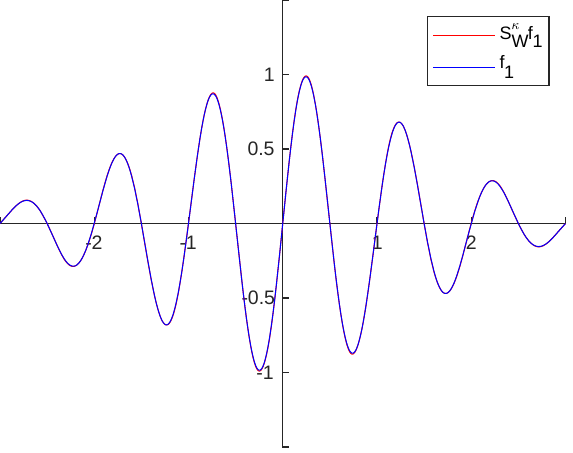}
         \caption{Graphs of $f_1$ and $S_W^{\kappa}f_1$ with $W=7.$}
         \label{q4w7f1}
     \end{subfigure}
      \hfill
     \begin{subfigure}[b]{0.3\textwidth}
         \centering
         \includegraphics[width=\textwidth]{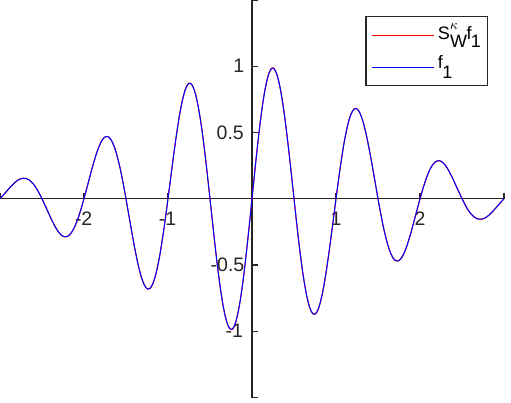}
         \caption{Graphs of $f_1$ and $S_W^{\kappa}f_1$ with $W=9.$}
         \label{q4w9f1}
     \end{subfigure}
        \caption{Approximation of $f_1$ based  on $S_W^{\kappa}f_1$  for $\kappa=(Q_4,0,3).$}
        \label{q4f1}
\end{figure}

\begin{figure}[H]
     \centering
     \begin{subfigure}[b]{0.3\textwidth}
         \centering
         \includegraphics[width=\textwidth]{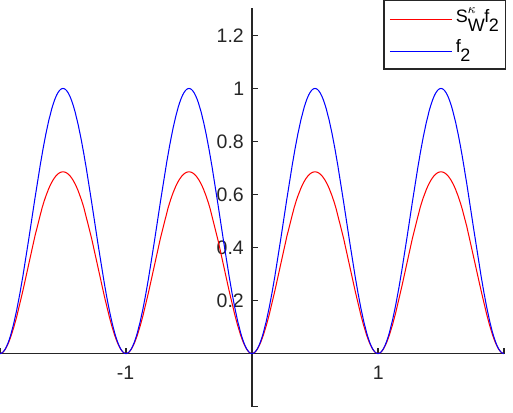}
         \caption{Graphs of $f_2,$ $S^1_Wf_2,$ and $S_W^{\kappa}f_2$ with $W=3.$}
         \label{q4w3f2}
     \end{subfigure}
     \hfill
     \begin{subfigure}[b]{0.3\textwidth}
         \centering
         \includegraphics[width=\textwidth]{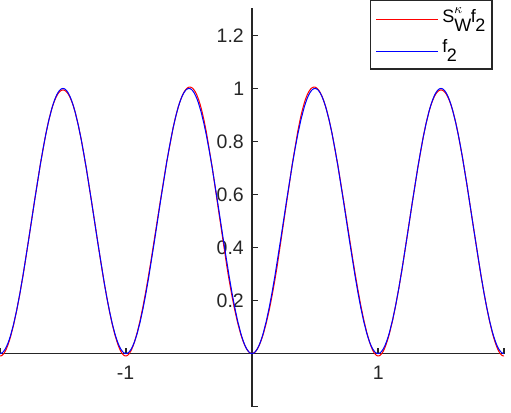}
         \caption{Graphs of $f_2$ and $S_W^{\kappa}f_2$ with $W=5.$}
         \label{q4w5f2}
     \end{subfigure}
     \newline
     \begin{subfigure}[b]{0.3\textwidth}
         \centering
         \includegraphics[width=\textwidth]{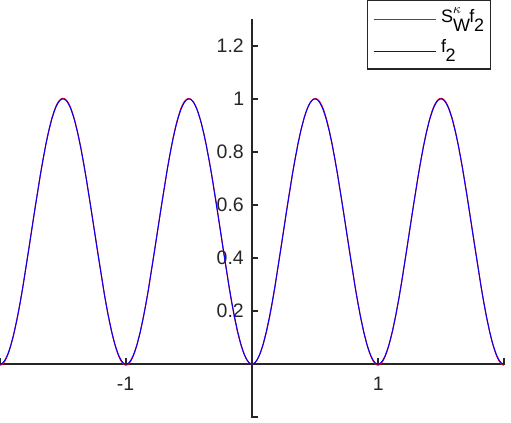}
         \caption{Graphs of $f_2$ and $S_W^{\kappa}f_2$ with $W=7.$}
         \label{q4w7f2}
     \end{subfigure}
     \hfill
     \begin{subfigure}[b]{0.3\textwidth}
         \centering
         \includegraphics[width=\textwidth]{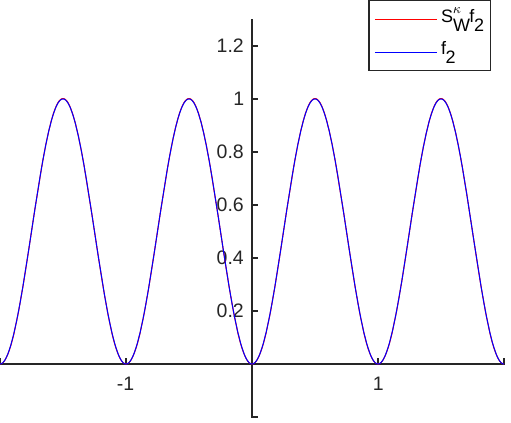}
         \caption{Graphs of $f_2$ and $S_W^{\kappa}f_2$ with $W=9.$}
         \label{q4w9f2}
     \end{subfigure}
        \caption{Approximation of $f_2$ based  on $S_W^{\kappa}f_2$  for $\kappa=(Q_4,0,3).$}
        \label{q4f2}
\end{figure}

\begin{figure}[H]
     \centering
     \begin{subfigure}[b]{0.3\textwidth}
         \centering
         \includegraphics[width=\textwidth]{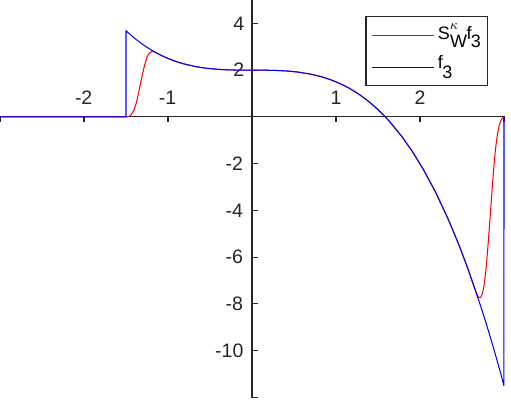}
         \caption{Graphs of $f_3$ and $S_W^{\kappa}f_3$ with $W=3\sqrt{7}.$}
         \label{q4w3sqrt7f3}
     \end{subfigure}
     \hfill
     \begin{subfigure}[b]{0.3\textwidth}
         \centering
         \includegraphics[width=\textwidth]{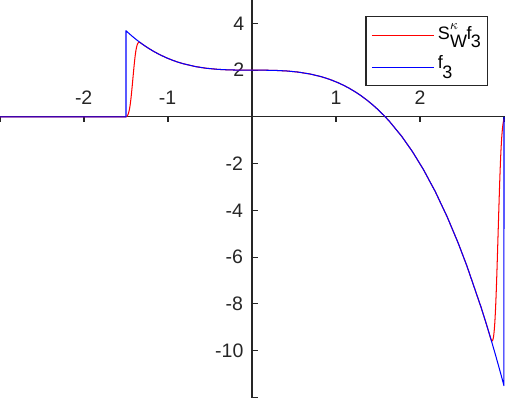}
         \caption{Graphs of $f_3$ and $S_W^{\kappa}f_3$ with $W=6\sqrt{7}.$}
         \label{q4w6sqrt7f3}
     \end{subfigure}
     \newline
     \begin{subfigure}[b]{0.3\textwidth}
         \centering
         \includegraphics[width=\textwidth]{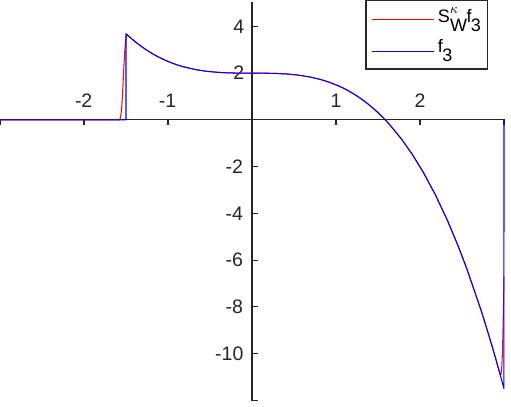}
         \caption{Graphs of $f_3$ and $S_W^{\kappa}f_3$ with $W=10\sqrt{7}.$}
         \label{q4w10sqrt7f3}
     \end{subfigure}
     \hfill
     \begin{subfigure}[b]{0.3\textwidth}
         \centering
         \includegraphics[width=\textwidth]{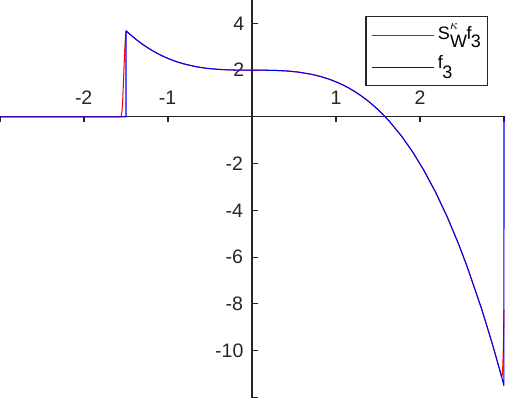}
         \caption{Graphs of $f_3$ and $S_W^{\kappa}f_3$ with $W=13\sqrt{7}.$}
         \label{q4w13sqrt7f3}
     \end{subfigure}
        \caption{Approximation of $f_3$ based  on $S_W^{\kappa}f_3$  for $\kappa=(Q_4,0,3).$}
        \label{q4f3}
\end{figure}

\section{Appendix}
We start with some results on Steklov functions $f_{r,h},$ which are well-known in connection with the classical moduli of continuity $\omega_r(f;x;\delta)$ as well as the $\tau$-modulus $\tau_r(f;\delta)_p.$ 
The following result from \cite{BBSV1} can be proved with a small modification.
\begin{prop}\label{frh}
For each function $f\in\Lambda_{\rho}^p,$ $1\leq p<\infty$ and each $r\in\mathbb{N},$ $h>0$, there exists  a function
$f_{r,h}$ with the following properties:
\begin{itemize}
\item [$(i)$] $f_{r,h}\in W_p^r\cap C(\mathbb{R})$ and for its $s^{th}$ derivative $$\|f_{r,h}^{(s)}\|_{L^p(\mathbb{R})}\leq C_rh^{-s}\tau_s(f;h)_p,~ (s=1, 2, \dots, r),$$
\item [$(ii)$] $|f^{(i)}(x)-f^{(i)}_{r,h}(x)|\leq \omega_r(f^{(i)};x;2h),~(x\in\mathbb{R})$, $i=0,1,\dots, \rho-1,$
\item [$(iii)$] $\|f-f_{r,h}\|_{L^p(\mathbb{R})}\leq \tau_r(f;2h)_p.$
\end{itemize}
\end{prop}
\begin{pf}[\underline{\textbf{Proof of Theorem \ref{geninterpolation}}}]
Arguing as in \cite{BBSV1}, we can easily show that 
\begin{equation}\label{omega}
\|\omega_r(f;\bullet;2h)\|_{\ell_{\rho}^p(\Sigma)}\leq 2^{\frac{1}{p}+2(r+1)}\sum\limits_{i=0}^{\rho-1}\tau_r\left(f^{(i)};h+\dfrac{\overline\Delta}{r}\right)_p,
\end{equation}
for every $f\in\Lambda_{\rho}^p$. 
Take the function $f_{r,h}\in W_p^r\cap C(\mathbb{R})$ of Proposition \ref{frh} for some $h>0.$ Then 
\begin{eqnarray*}
\|L_{\beta}f-f\|_{L^p(\mathbb{R})}
&\leq&\|L_{\beta}(f-f_{r,h})\|_{L^p(\mathbb{R})}+\|L_{\beta}f_{r,h}-f_{r,h}\|_{L^p(\mathbb{R})}+\|f_{r,h}-f\|_{L^p(\mathbb{R})}\\
&=:& S_1+S_2+S_3,
\end{eqnarray*}
say. By Proposition \ref{frh} and \eqref{omega}, we have
\begin{eqnarray}\label{S1}
S_1
&\leq& K_1\|f-f_{r,h}\|_{\ell_{\rho}^p(\Sigma_{\beta})}\nonumber\\
&=&K_1\left\{\sum\limits_{i=0}^{\rho-1}\sum\limits_{j\in\mathbb{Z}}\left|\left(f^{(i)}(x_{j,\beta})-f^{(i)}_{r,h}(x_{j,\beta})\right)\right|^p\Delta_{j,\beta}\right\}^{1/p}\nonumber\\
&\leq&K_1\Big\{\sum\limits_{i=0}^{\rho-1}\sum\limits_{j\in\mathbb{Z}}|\omega_r(f^{(i)};x_{j,\beta};2h)|^p\Delta_{j,\beta}\Big\}^{1/p}\nonumber\\
&\leq& K_12^{\frac{1}{p}+2(r+1)}\sum\limits_{i=0}^{\rho-1}\tau_r\left(f^{(i)};h+\dfrac{\overline\Delta_{\beta}}{r}\right)_p.
\end{eqnarray}
Take $h= \overline\Delta_{\beta}^{s/r}.$ Since $\overline\Delta_{\beta}\leq r$ and $s\leq r,$
$$h+\frac{\overline\Delta_{\beta}}{r}\leq 2\overline\Delta_{\beta}^{s/r}.$$
Therefore, \eqref{S1} becomes
\begin{eqnarray}\label{interpolations1}
S_1&\leq&K_12^{\frac{1}{p}+2(r+1)}\sum\limits_{i=0}^{\rho-1}\tau_r(f^{(i)};2\overline\Delta_{\beta}^{s/r})_p\nonumber\\
&\leq&K_12^{\frac{1}{p}+2(r+1)}4^{r+1}\sum\limits_{i=0}^{\rho-1}\tau_r(f^{(i)};\overline\Delta_{\beta}^{s/r})_p.
\end{eqnarray}
From Proposition \ref{frh}, we obtain 
\begin{equation}\label{interpolations2}
S_2\leq K_2\overline\Delta_{\beta}^s\|f_{r,h}^{(r)}\|_{L^p(\mathbb{R})}\leq K_2 C_r\tau_r(f;\overline\Delta_{\beta}^{s/r})_p
\end{equation}
and 
\begin{eqnarray}\label{interpolations3}
S_3\leq\tau_r(f;2h)_p
\leq4^{r+1}\tau_r(f;\overline\Delta_{\beta}^{s/r})_p,
\end{eqnarray}
using \eqref{proptau}. Altogether inequalities \eqref{interpolations1}, \eqref{interpolations2}, and \eqref{interpolations3} imply  \eqref{interpolationthm}.
\end{pf}

\end{document}